\documentclass[12pt,a4paper]{article}
\usepackage[T2A]{fontenc}
\usepackage[cp1251]{inputenc}
\usepackage[russian, english]{babel}
\usepackage{epsfig,amssymb,color,graphics,epstopdf}
\usepackage{amsthm, amsmath, amssymb, amsbsy, amsfonts,  latexsym, euscript}
\usepackage{tikz}
\usepackage{caption}
\usetikzlibrary{arrows,%
                topaths}%
\usepackage{tkz-berge}
\usepackage{verbatim}
cm \hoffset = -1.0 true cm

\usepackage{titlesec}
\titleformat*{\section}{\large\bfseries}
\titleformat*{\subsection}{\normalsize\bfseries}
\titleformat*{\subsubsection}{\normalsize\bfseries}

\usepackage{csquotes}
\begin{document}

\newtheorem{teo}{Theorem}
\newtheorem*{teon}{Theorem}
\newtheorem{lem}{Lemma}
\newtheorem*{lemn}{Lemma}
\newtheorem{prp}{Proposition}
\newtheorem*{prpn}{Proposition}
\newtheorem{ass}{Assertion}
\newtheorem*{assn}{Assertion}
\newtheorem{assum}{Assumption}
\newtheorem*{assumn}{Assumption}
\newtheorem{stat}{Statement}
\newtheorem*{statn}{Statement}
\newtheorem{cor}{Corollary}
\newtheorem*{corn}{Corollary}
\newtheorem{hyp}{Hypothesis}
\newtheorem*{hypn}{Hypothesis}
\newtheorem{con}{Conjecture}
\newtheorem*{conn}{Conjecture}
\newtheorem{dfn}{Definition}
\newtheorem*{dfnn}{Definition}
\newtheorem{problem}{Problem}
\newtheorem*{problemn}{Problem}
\newtheorem{notat}{Notation}
\newtheorem*{notatn}{Notation}
\newtheorem{quest}{Question}
\newtheorem*{questn}{Question}

\newtheorem{rem}{Remark}
\newtheorem*{remn}{Remark}
\newtheorem{exa}{Example}
\newtheorem*{exan}{Example}
\newtheorem{cas}{Case}
\newtheorem*{casn}{Case}
\newtheorem{claim}{Claim}
\newtheorem*{claimn}{Claim}
\newtheorem{com}{Comment}
\newtheorem*{comn}{Comment}

\newenvironment{proofn}{{\bf Proof: }}
{\hfill $\diamond$\medskip}

\title{Three Heteroclinic Orbits Induce a Countable Family of Equivalence Classes of Regular Flows}%Full Title article
\author{Elena~Gurevich}%
\date{}

%\pacs{37C15}%MSC2010 -

\maketitle

%Here You may place your own definition

\newcommand{\Int}{{\rm int}\,}
\newcommand{\Cl}{{\rm cl}\,}
\newcommand{\Cr}{{\rm Cr}}
\newcommand{\Dim}{{\rm dim}\,}
\newcommand{\Id}{{\rm id}}
\newcommand{\IM}{{\rm Im}\,}
\newcommand{\KEr}{{\rm Ker}\,}
\newcommand{\Lk}{{\rm lk}\,}

\sloppy

\abstract{In the paper,  a class of smooth structurally stable flows on closed four-dimensional manifolds, the non-wandering set of which  contains exactly two saddle equilibria, and the wandering set contains isolated trajectories connecting these saddle equilibria (heteroclinic curves) is considered. For such flows, we describe the topology of the ambient  manifold,  provide necessary and sufficient conditions of topological equivalence, and describe an algorithm of building a representative of any class of topological equivalence. In particular, we show that for a flow of the class under consideration, given  on the complex projective plane $\mathbb{CP}^2$, the number of heteroclinic curves is a complete topological invariant. At the same time,  we show that for  every fixed  odd integer  $\gamma\geq 3$,  there exist  countably many  equivalence classes of flows on the sphere $\mathbb S^4$ with exactly $\gamma$  heteroclinic curves. These results contrast with a  three-dimensional case, where under similar conditions, for  any $\gamma$  there exists only finite set of equivalence classes.
}

\section*{INTRODUCTION}%you can use numbered INTRODUCTION without symbol *
%\textit{RCD} 
%\begin{itemize}\itemsep=-2pt

 The problem of topological classification of gradient-like flows, that is structurally stable flows with finite nonwandering set,  has a long and rich history, originating in the classical works of A. A. Andronov, L. S. Pontryagin, E. A. Leontovich, and A. G. Mayer~\cite{AP37,LM37,LM55}. For gradient-like flows on closed manifolds of dimensions 2 and 3, this problem was completely solved in~\cite{LM55, Pe, Um90}. For the  dimension of the  ambient  manifolds equal to  four  and higher, the most complete results have been obtained only for gradient-like flows without  heteroclinic  trajectories, connecting saddle equilibria  (see~\cite{Pil, MeZh13, MeZh16, GrGu22, GrGu23, GuSa24, GuSa25, GuSa26}). The goal of this paper is to study the effects of  heteroclinic curves on the asymptotic behavior of trajectories of gradient-like flows  on four-dimensional manifolds. As recent works~\cite{GrGu22, GuSa24,  GuSa26} show,  dimension 4 is transitional for gradient-like flows. On the one hand,  gradient-like flows on such manifolds, similar to Morse-Smale diffeomorphisms on three-dimensional manifolds,   can have wildly embedded closures of two-dimensional invariant manifolds. This effect leads, in particular, to the fact that the class of flows with a nonwandering set containing only two saddle equilibria and having no heteroclinic intersections consists of a  countable set of topologically nonequivalent flows, and equivalence  classes  cannot be distinguished by traditional combinatorial invariants. On the other hand, the dynamics of such flows is closely related to the embedding of traces of invariant manifolds in a three-dimensional cross-section. This fact allows  to use three-dimensional topology, in particular, knot theory and Dehn surgery,   for  studying such systems.  Due to this, it is possible not only to obtain  abstract necessary and sufficient conditions of topological equivalence of the flows, but to solve a realization problem and  discover meaningful classes of dynamical systems that admit a simple description but  exhibit an interplay between the asymptotic behavior of trajectories, the topology of the underlying manifold, and the embedding of invariant manifolds. One of such classes is studied in this paper. 
 
 In accordance with the reviewer's advice, the exact statements of the results are preceded by their informal versions and main ideas of the proof. All necessary definitions, facts and references  will be given  in further sections of  the  paper.   We consider  gradient-like  flows  on closed four-dimensional manifolds, the non-wandering set of which  contains exactly two saddle equilibria, and the wandering set contains isolated trajectories connecting these saddle equilibria (heteroclinic curves).  For dimension  reason, such  heteroclinic curves must belong to the intersection of invariant manifolds of dimensions two and three.  As an immediate consequence of the existence of an  {\it energy function} for gradient-like flows, we obtain  that the closures of invariant manifolds of  saddle equilibria  that  do not contain heteroclinic curves  form connected attractor and repeller of the flow. Moreover, this attractor and repeller  can be separated from each other by a  smoothly embedded closed  three-dimensional manifolds $\Sigma$ ({\it a characteristic cross-section})  that intersects transversely each trajectory moving from  the repeller to the attractor. We show that  invariant manifolds of the saddles  that do contain the heteroclinic curves,  intersect $\Sigma$ along smoothly embedded two- and one-dimensional spheres $\Lambda_s, \lambda_u$.   We show in Theorem~\ref{cond} that the  topology of $\Sigma$ and the  embedding of this two spheres completely determine the flow  belonging to  the class under consideration, up to topological equivalence. Then we solve the problem of  building   a canonical representative  in each class of topological equivalence. For this, we  study  all possible types of  $\Sigma$ and the embedding of  traces  of the invariant manifolds in it. Its    turns out that $\Sigma$ can be diffeomorphic only either to the sphere $\mathbb S^3$ or to $\mathbb S^2\times \mathbb S^1$. As Theorem~\ref{top} shows, the ambient manifold of the flow is homeomorphic to the complex projective plane $\mathbb{CP}^2$ in the first case and to the four-dimensional sphere $\mathbb S^4$ in the second case.

 We show in Lemma~\ref{seq1} that each of the  traces $\Lambda_s, \lambda_u$ of the invariant manifolds  must have a unique position in $\Sigma$ up to an ambient homeomorphism. In particular, $\lambda_u$ must be {\it trivial knot}, that means that a complement in $\Sigma$ to an open tubular neighborhood of $\lambda_u$ is homeomorphic to a solid torus. This fact is quite surprising and based on non-trivial theorems of Gordon and Luecke.  In case   when $\Sigma$ is diffeomorphic to $\mathbb S^3$, these observations turn out to be enough to prove that only the number of the intersection points of $\Lambda_s, \lambda_u$ (each of which generates the heteroclinic curve) determines the class of topological equivalence of the corresponding flow (Theorem~\ref{clas+}).    One more  surprise is that  one may construct a countable set of pairwise nonequivalent pairs consisting of a  2-sphere and a trivial knot in  $\mathbb S^2\times \mathbb S^1$ (they  are shown in the Figure~\ref{diff}). So,  trivial knots in $\mathbb S^2\times \mathbb S^1$ may become distinguishable up to a homeomorphism which must preserve a sphere that does not separate $\mathbb S^2\times \mathbb S^1$.  This fact leads, in virtue of  Theorem~\ref{cond},  to the countable set of pairwise topologically nonequivalent flows on $\mathbb S^4$ with the same number of saddles and heteroclinic curves (Theorems~\ref{clas},~\ref{realiz}).

 \section{STATEMENTS OF RESULTS}\label{str}

 Let  $M^n$ denote a smooth closed manifold of dimension~$n\geq 1$. Recall that an equilibrium state $p$ of a smooth flow $f^t$ on $M^n$ is called {\it hyperbolic} if the Jacobi matrix $J_p=\left.\left(\frac{\partial F}{\partial x}\right)\right|_{p}$ of the   velocity field $F=\left.\frac{\partial f^t(x)}{\partial t}\right|_{{t=0}}$ at the point $p$ has no eigenvalues with zero real part. The sets $$W^s_p=\{x\in M^n: \lim\limits_{t\to +\infty} f^{t}(x)= p\},$$ $$W^u_p=\{x\in M^n: \lim\limits_{t\to +\infty} f^{-t}(x)= p\}$$ are called {\it stable and unstable invariant manifolds of the equilibrium state $p$}, respectively. The number $\mu_p$ equal to the dimension $\dim W^u_{p}$ of the unstable manifold $W^u_p$ is called {\it the  Morse index of $p$}. If  $\mu_p=0$ (or $\mu_p=n$) then $p$ is called {\it a sink {\rm(}a source, respectively{\rm)}}, and  if $\mu_p\in (0,n)$, then $p$ is called {\it a saddle}.

A smooth flow $f^t$ on $M^n$ is called {\it gradient-like} if its nonwandering set is finite and consists of hyperbolic equilibria, and the invariant manifolds of the equilibria intersect each other  transversely. If $p,q$ are saddle equilibria such that $W^s_p\cap W^u_q\neq \emptyset$, we will call the intersection $W^s_p\cap W^u_q$ and any trajectory in it  {\it heteroclinic}.  

Let $G_{\mu,\nu,k}$ denote the class of gradient-like flows with  closed four-dimensional carrying  manifolds and such that the set  of saddle equilibria of any $f^t\in G_{\mu,\nu,k}$  consists of exactly two points with Morse indices $\mu, \nu\in \{1,2,3\}$, $\mu\leq \nu$. The number $k$ denotes the total number of equilibria of $f^t$. The following theorem describes all possible combinations of the numbers $\mu, \nu, k$ and the topology of the manifold $M^4$ admitting flows from the corresponding class.

Everywhere below,  $\mathbb S^n$ denotes a unit sphere  $\{x\in \mathbb R^{n+1}|\, |x|\leq 1\}$ viewed as a smooth submanifold of $\mathbb R^{n+1}$ with a canonical smooth structure and orientation. We also use a  term {\it sphere} and a notation $S^n$ for a manifold  homeomorphic to the unit sphere $\mathbb S^n$ (for instance, for an image of a topological  embedding of  $\mathbb S^2$ into  $M^4$). We remark  that, due to~\cite{Moise, Munk} \footnote{Moise proved in~\cite[Theorems~3,4]{Moise} that any closed three-dimensional manifold admits a triangulation, more over, any homeomorphic three-dimensional manifolds  also are  PL-homeomorphic. Munkres  shown in~\cite[Theorem 6.3]{Munk} that there are no obstruction for smoothing a  PL homeomorphism of three-dimensional manifolds.}, for $n \leq 3$, any  homeomorphic closed manifolds are diffeomorphic,  whereas for $n=4$, even in the case of spheres $S^4$, this question remains open and is known as the smooth Poincare conjecture.

\begin{teo} \label{top} Let $f^t\in G_{\mu,\nu,k}$ and $M^4$ is the ambient manifold of $f^t$. Then exactly one of the following combinations holds.

\begin{enumerate}\itemsep=-2pt
\item $k=6, \mu, \nu\in \{1,3\}$ or $k=4$, $\mu\in \{1,2\}, \nu=\mu+1$, and $M^4$ is homeomorphic to  $\mathbb S^4$.

\item $k=5$, $\mu\in \{1,2\}, \nu=\mu+1$, and $M^4$ is homeomorphic to the complex projective plane $\mathbb{CP}^2$.

\item $k=4, \mu=1, \nu=3$, and $M^4$ is homeomorphic to either the direct product $\mathbb S^3\times \mathbb S^1$, or to the unique non-orientable $\mathbb S^3$-bundle over $\mathbb S^1$.
\item $k=4, \mu=\nu=2$, and $M^4$ is homeomorphic to one of the following manifolds: the connected sum $\mathbb{CP}^2\sharp \mathbb{CP}^2$, the direct product $\mathbb S^2\times \mathbb S^2$, or a non-trivial $\mathbb S^2$-bundle over $\mathbb S^2$.

\end{enumerate}
\end{teo}

Each of the classes described in Theorem~\ref{top}, except $G_{1,2,4}\cup G_{2,3,4}$, contains flows without heteroclinic intersections. Moreover, classes $G_{1,1,6}, G_{3,3,6}$ and $G_{2,2,4}$ entirely consist of flows without heteroclinic intersections. Topological classification of  flows without heteroclinic  intersections in all possible  classes $G_{\mu, \nu, k}$ follows  from~\cite{Pil, GrGu23, GrGu22, GuSa24}. Below, in Theorems~\ref{clas+},\,\ref{clas} and \ref{realiz}, we obtain necessary and sufficient conditions for the topological equivalence of flows from the classes $G_{1,2,k}$,  and describe an algorithm for constructing a representative of each  class of topological  equivalence. The topological classification for $G_{2,3,k}$ reduces to Theorems~\ref{clas+}, \ref{clas} by time reversing. In other cases, the problem of topological classification can be solved using the same methods and is of no particular interest.

\begin{teo}\label{clas+} Let  $f^t\in G_{1,2,5}$. If the set of  heteroclinic intersections of $f^t$ is non-empty, then it consists of an even number of heteroclinic trajectories. Flows $f^t, {f'}^t\in G_{1,2,5}$ are topologically equivalent if and only if they either have no heteroclinic intersections or have the same number of heteroclinic trajectories. For any integer $\gamma\geq 0$, there exists a flow $f^t\in G_{1,2,5}$ that has exactly $2\gamma$ heteroclinic trajectories.
\end{teo}

Theorem~\ref{clas+} states  that  the class of topological equivalence of  a flow $f^t\in G_{1,2,5}$ is completely determinated by the number $\gamma$ of heteroclinic trajectories. We show  that for   $G_{1,2,4}$ such result is true only if the number of heteroclinic curves  is equal to   one, while topological   classification of  flows from $G_{1,2,4}$  with greater number of heteroclinic intersections requires  more subtle invariants.  We   use the following invariant, which is  convenient to  solution of the realization problem.    Due to Theorem~\ref{top}, the ambient manifold $M^4$  of $f^t\in G_{1,2,4}$ is homeomorphic to  $\mathbb S^4$.  Let $\sigma_1, \sigma_2$ be saddle equilibria of $f^t$ with Morse indices equal to 1 and 2, respectively, and $$\Sigma\subset M^4$$ be a smooth closed three-dimensional submanifold of $M^4$  that separates the sets $A={\rm cl}\,W^u_{\sigma_1}$ and $R={\rm cl}\,W^s_{\sigma_2}$,  and transversely intersects each trajectory belonging to  $M^4\setminus (A\cup R)$. We  call $\Sigma$ {\it a characteristic cross-section} of the flow $f^t$. Set $$\Lambda_s=\Sigma\cap W^s_{\sigma_1}, \lambda_u=\Sigma\cap W^u_{\sigma_2}.$$ 

We refer to  a  triple $\{\Sigma, \Lambda_s, \lambda_u\}$  as  {\it a scheme of a flow $f^t$}. We show in Lemma~\ref{seq1} that for any flow $f^t\in G_{1,2,4}$ the manifold $\Sigma$ is diffeomorphic to the direct product $\mathbb S^2\times \mathbb S^1$, $\Lambda_s$ is a homotopically  nontrivial two-dimensional sphere in $\Sigma$, and $\lambda_u$ is a trivial knot, i.e., a smoothly embedded simple closed curve such that $g(\lambda_u)=\{x\}\times S^1$ for some diffeomorphism $g: \Sigma\to \mathbb S^2\times \mathbb S^1$.  

To compare the flows $f^t, {f'}^t\in G_{1,2,4}$ by means of their schemes, we will denote by $\{\Sigma', \Lambda_s', \lambda_u'\}$ a scheme of the flow ${f'}^t$ meaning that manifolds $\Sigma', \Lambda_s', \lambda_u'$  are chosen similar to $\Sigma, \Lambda_s, \lambda_u$, respectively. We say that  schemes $\{\Sigma, \Lambda_s, \lambda_u\}$, $\{\Sigma', \Lambda_s', \lambda_u'\}$ of flows $f^t, {f'}^t$ are {\it equivalent} if there exists a homeomorphism $h: \Sigma\to \Sigma'$ such that $h(\Lambda_s)=\Lambda_s'$, $h(\lambda_u)=\lambda_u'$. In Proposition~\ref{indep}  we show that any two schemes of the same flow $f^t$ are equivalent, so  $f^t$ determines a class of equivalence of its scheme. The following theorem shows that the class of  equivalence of the  scheme completely defines a flow $f^t$  up to the topological equivalence.

\begin{teo}\label{clas} Any flow $f^t\in G_{1,2,4}$ has a nonempty set of heteroclinic curves consisting of an odd number of curves. Flows $f^t, {f'}^t\in G_{1,2,4}$ having a unique heteroclinic trajectory, are topologically equivalent.  Flows $f^t, {f'}^t\in G_{1,2,4}$ that have more than one heteroclinic trajectory are topologically equivalent if and only if their schemes are equivalent.
\end{teo}

Let $\Lambda, \lambda \subset \mathbb S^2\times \mathbb S^1$ be  smoothly embedded homotopically nontrivial two-dimensional sphere and the trivial knot, respectively, that intersect each other transversely. We refer to  a triple $\{\mathbb S^2\times \mathbb S^1, \Lambda, \lambda\}$ as an {\it abstract scheme}. 

\begin{teo}\label{realiz}  For any abstract scheme $\{\mathbb S^2\times \mathbb S^1, \Lambda, \lambda\}$,  there exists a flow $f^t\in G_{1,2,4}$ whose scheme is equivalent to  $\{\mathbb S^2\times \mathbb S^1, \Lambda, \lambda\}$. 

In particular, for any integer $\gamma \geq 1$, the class $G_{1,2,4}$ contains a countable set of topologically nonequivalent flows  with exactly $(2\gamma+1)$ heteroclinic trajectories.
\end{teo}

\section{TOPOLOGICAL ASPECTS OF THE DYNAMICS OF FLOWS FROM CLASS $G_{\mu, \nu, k}$}\label{top_s}

In this section, we prove Theorem~\ref{top} and Lemma~\ref{seq1}, which describe the  topology of the  embedding of  closures of invariant manifolds of saddle equilibria.

\subsection{Energy Function}

Recall that a $C^2$-smooth function $\varphi: M^n\to \mathbb R$ is called a {\it Morse function} if all its critical points are non-degenerate. The {\it Morse index} of a non-degenerate critical point $p\in M^n$ is the number (counting with multiplicity) of negative eigenvalues of the Hessian matrix evaluated at this point.
In~\cite{Sm61,Me68}, it was proven that for any gradient-like flow  there exists {\it an energy function}, that is  a Morse function  strictly decreasing along  trajectories other than equilibrium states and that has  a critical point at each equilibrium state. This implies the following assertion, which refines the properties of the energy function for flows in the class under consideration.

\begin{prp}\label{SM}
For any flow $f^t\in G_{\mu, \nu, k}$, there exists an energy function, i.e., a Morse function $\varphi: M^4\to [0,4]$ with the following properties:
\begin{enumerate}
\item $\varphi$ is strictly decreasing along flow trajectories other than equilibrium states;
\item the trajectories of the flow $f^t$ transversely intersect the set $\varphi^{-1}(c)$ for any non-critical value $c\in (0,4)$ of the function $\varphi$;
\item the set of critical points of the function $\varphi$ coincides with the set $\Omega_{f^t}$;
\item for any equilibrium $p\in \Omega_{f^t}$ 
there exists a neighborhood $U_p$ equipped with local coordinates {\rm(}Morse coordinates{\rm)} $y_1,\dots, y_4$ such that the sets $W^u_p\cap U_p, W^s_p\cap U_p$ are described by the equations
$y_{\mu_p+1}=\dots=y_4=0$, $y_1=\dots=y_{\mu_p}=0$, respectively, and the function $\varphi$ has the form 
$$\varphi(y_1,\dots, y_4)=\mu_p+\sum\limits_{i=1}^4\alpha_iy_i^2,$$
where $\alpha_i\in \{+1, -1\}$ and the number of terms with negative $\alpha_i$ is equal to $\mu_p={\rm dim}\,W^u_p$.
\end{enumerate}
\end{prp}

\subsection{Proof of Theorem~\ref{top}}

We provide the proof of the theorem for the case $\mu=1,\nu=2$, which can be easily generalized to the remaining cases except for $\mu=\nu=2$ and the cases where $k=4,\mu=1,\nu=3$. The theorem for these last cases was proven in~\cite{GuSa24,GuSa25}. 

Let $f^t\in G_{1, 2, k}$, $\sigma_1, \sigma_2$ be the saddle equilibria of the flow $f^t$ of Morse indices 1, 2, respectively, $M^4$ be the supporting manifold, and $\varphi: M^4\to [0,4]$ be an energy function of the flow $f^t$ satisfying the conclusion of Proposition~\ref{SM}.

 First, let us determine  a possible value of   $k$. Since the function $\varphi$ decreases along the trajectories of the flow $f^t$, the manifold $M^4$ can be represented as a  union of stable manifolds of all equilibria of the flow $f^t$. Let us denote by $X$ the union of the stable manifolds of all sinks, by $Y$ the  union of $W^s_{\sigma_2}$ and  all  sources. Then $M^4=X\cup Y\cup W^s_{\sigma_1}$.  Due to~\cite[Theorem 8,9]{HW},  if  $M^n\setminus Y$  is disconnected, then ${\rm dim} Y=(n-1)$. So, since ${\rm dim} W^s_{\sigma_2}=2$, then $M^4\setminus Y$ is connected. Let us consider the set $X=(M^4\setminus Y)\setminus W^s_{\sigma_1}$. Due to~\cite[Lemma 3.8]{Sm60},   the stable manifold of any  equilibrium state $p$ of  $f^t$  is a smoothly embedded open ball of dimension $(n-\mu_p)$.  By the definition of smooth embedding,   for any point $x\in W^s_{\sigma_1}$ there is a neighborhood $U_x\subset M^4\setminus Y$  and a homeomorphism $h: U_x\to \mathbb R^4$ such that  the set  $h(U_x\cap W^s_{\sigma_1})$ is a three-dimensional linear  subspace. That is why $U_x\setminus W^s_{\sigma_1}$ has  two connected components.  Then the set $X=(M^4\setminus Y)\setminus  W^s_{\sigma_1}$   either is  connected or has two  connected components.   Consequently, only one of the following two possibilities is realized: 1) the flow $f^t$ has a unique sink $\omega$; 2) the flow $f^t$ has two sinks $\omega_1, \omega_2$. Moving on to the function $\psi=4-\varphi$, which is the energy function for the flow $f^{-t}$, and taking in account that the nonwandering set of $f^{-t}$ does not contain saddle points of Morse index 1,  we similarly find that the set of equilibrium states of the flow $f^t$ contains exactly one source $\alpha$. Then $k=4$ in case 1) and $k=5$ in case 2).

To determine the topological type of  $M^4$, we remark that $M^4$  is simply connected, that is its fundamental group $\pi_1(M^4)$ is trivial. This fact  directly  follows from~\cite[Corollary 10.1.8, p. 225]{Pa88}  where it is shown that if $M^n$ is not simply connected, then every Morse function on $M^n$ has a critical point of index 1. Since a Morse function $\psi=4-\varphi$ has no such critical points, $M^4$ is simply connected.

Let $\beta_i$ denote the $i$-th Betti number of the manifold $M^4$ (the rank of homology group $H_i(M^4, \mathbb Z)$). Since $M^4$ is simply connected, $\beta_0=1, \beta_1=0$. It follows from Poincare duality  that $\beta_4=\beta_0=1, \beta_3=\beta_1=0$. By the Morse equality (see~\cite[\S 5]{Morse}) we have $$k_0-k_1+k_2-k_3+k_4=\beta_0-\beta_1+\beta_2-\beta_3+\beta_4.$$ Then in case 1) $\beta_2=0$ and $M^4$ is a simply connected homology sphere.  By four-dimensional Poincare Theorem, that follows from Freedman's results~\cite[Theorem~1.6]{Fr} (we recall the statements of these results  in Statement~\ref{Fr-a} and Corollary~\ref{Fr-a1} in Section~\ref{tm} of this paper),  $M^4$ is homeomorphic to the sphere $\mathbb S^4$. In case 2) we obtain $\beta_2=1$. Once more applying~\cite[Theorem~1.6]{Fr} (see Corollary~\ref{Fr-a2} in Section~\ref{tm}), we obtain that there exists a unique (up to homeomorphism) simply connected smooth manifold with such homology groups,  and this manifold is the complex projective plane $\mathbb{CP}^2$. Theorem~\ref{top} is proven.

\begin{notat}\label{nw} In the proof of Theorem~\ref{top}, it was shown that the nonwandering set of the flow $f^t\in G_{1,2,k}$, in addition to  saddles $\sigma_1, \sigma_2$, contains exactly one source and one {\rm (}two{\rm )} sinks in the case $k=4$ {\rm (}resp., $k=5${\rm)}.

\end{notat}

\subsection{A handle decomposition of the ambient manifold}\label{dec}

Everywhere below, unless otherwise stated, an {\it $n$-dimensional ball or merely  $n$-ball} is a manifold $B^n$ homeomorphic to the unit ball $$\mathbb B^n=\{x\in \mathbb R^n|\, |x|\leq 1\},$$   an open $n$-ball and an {\it $(n-1)$-sphere} $S^{n-1}$ are manifolds homeomorphic to the interior ${\rm int}\, \mathbb B^n$ and the boundary $\partial \mathbb B^n$ of  $\mathbb B^n$, respectively.

We recall that a {\it handle of index $i$} or just  {$i$-handle} is the direct product $$H^n_i=\mathbb {B}^i\times \mathbb {B}^{n-i}.$$ The boundary of the handle is naturally represented as the union of two subsets $$P_i=\partial \mathbb B^i\times \mathbb B^{n-i}, \, Q_i=\mathbb B^i\times \partial\mathbb B^{n-i}$$ with a common boundary. The set $P_i$ is called a {\it foot of the handle}. Spheres $\partial{\mathbb B}^i\times \{O\}$ and $\{O\}\times \partial{\mathbb B}^{n-i}$ are called the {\it foot  and belt spheres}, respectively.

Let $X$ be an $n$-dimensional compact manifold with boundary, $\psi: P_i \to \partial X$ be a smooth or topological   embedding, and $$Y=X\cup_\psi H^n_i$$ be a manifold obtained by identifying points $x\in P_i$ and $\psi(x)$. We say that $Y$ is obtained by {\it gluing the handle $H^n_i$} to the manifold $X$.
Representing a manifold as a union of handles that are glued consequently as it described above, is called its {\it a  handle decomposition of the manifold} (see, for instance,~\cite[Ch.\,6]{RuSa}).

A Morse function $\varphi: M^n\to M^n$ defines a handle decomposition  of $M^n$ (see, for instance, \cite[Theorems 3.2, 3.4]{M}) such that each $i$-handle contains a unique critical point of index $i$. This fact, together with Proposition~\ref{SM} and Remark~\ref{nw}, implies the following Proposition~\ref{hd}, which describes the handle decomposition of the ambient  manifold $M^4$ of a flow $f^t\in G_{1,2,k}$,  determined by its energy function $\varphi$.

Let us recall that,  due to statement 4 of  Proposition~\ref{SM}, the value of $\varphi(p)$ at any equilibrium state $p$ of $f^t$ equals to the Morse index $\mu_p$ of $p$. Hence, $0, 1, 2$ and $4$ are the critical values of $\varphi$. 
For $c\in [0,4]$, we set $M_c=\varphi^{-1}([0,c])$, $\Sigma_c=\partial M_c$.  
We choose numbers $c_1, c_2, c_3$ such that
\begin{equation}\label{const}
0<c_1<1<c_2<2<c_3<4.
\end{equation}

\begin{prp}\label{hd}
\begin{enumerate}
\item The sets $M_{c_1}$ and $M^4\setminus {\rm int}\,  M_{c_4}$ are the unions of 0-handles and a 4-handle, respectively.
\item For each $i\in \{1,2\}$, the set $M_{c_{i+1}}$ is obtained from $M_{c_{i}}$ by gluing the $i$-handle.

\item The manifold
$\Sigma_{c_2}$ is diffeomorphic to $\mathbb S^2\times \mathbb S^1$ for $k=4$ and to the sphere $\mathbb S^3$ for $k=5$.
\end{enumerate}
\end{prp}
\begin{proofn}
Applying Proposition~\ref{SM} and Theorem~\ref{top},  taking into account Remark~\ref{nw}, we get  that the energy  function $\varphi$ has a unique maximum point and one (two) minimum points in the case $k=4$ (resp., $k=5$). Therefore, the manifold $M^4$ is obtained from $M_{c_3}$ by gluing a single handle $M^4\setminus {\rm int}\, M_{c_3}$ of index 4. It follows that $\Sigma_{c_3}$ is a $3$-sphere and the manifold $M_{c_3}$ is connected. On the other hand, $M_{c_3}$ is obtained from $M_{c_2}$ by gluing a handle of index 2. Since gluing handles of index 2 does not change the number of connected components of the manifold, the manifold $M_{c_2}$ is also connected. 

Now consider  case $k=4, 5$ separately. In case $n=4$, the manifold $M_{c_2}$ is obtained form $0$-handle  $M_{c_1}$ by gluing a $1$-handle $H^4_1$ by means of some diffeomorphism $\psi$. Due to Statement~\ref{bundl} in Section~\ref{tm},  $M_{c_2}$ is homeomorphic either to the direct product $\mathbb B^n\times \mathbb S^1$ or to a non-orientable $B^n$-bundle over $S^1$. By Theorem~\ref{top}, the manifold $M^4$ is homeomorphic to $\mathbb S^4$, and hence is orientable. Therefore, $M_{c_2}$ is also orientable and, consequently,  is    homeomorphic to $\mathbb B^3\times \mathbb S^1$. In its turn,    $\Sigma_{c_2}=\partial M_{c_2}$  is homeomorphic to $\mathbb S^2\times \mathbb S^1$. 

In the case $k=5$, the manifold $M_{c_2}$ is a connected union of two 0-handles and one 1-handle, and hence is homeomorphic to the ball $\mathbb B^4$. Thus, in this case $\Sigma_{c_2}$ is homeomorphic to the sphere $\mathbb S^3$. As it was already mentioned in Section~\ref{str}, due to~\cite{Moise},~\cite{Munk},    homeomorphic closed manifolds of dimension three are also diffeomorphic, that gives the  required statements.
\end{proofn}

\subsection{Properties of the  scheme of the flow  $f^t\in G_{1,2,k}$}\label{propsec}

In this section, we refine the definition of the scheme of the  flow given  in the introduction for $f^t\in G_{1,2,4}$, extend it  for flows of class $G_{1,2,5}$, and clarify some properties of the schemes.

\begin{dfn}\label{seq}
We say that a closed three-dimensional topological submanifold $\Sigma\subset M^4$ is a cross-section for a flow $f^t$ defined on $M^4$ if there exists a topological embedding $e: \Sigma\times [-1,1]\to M^4$ such that $e(\Sigma\times \{0\})=\Sigma$
and for any point $x\in \Sigma$, the segment $e(\{x\}\times [-1,1])$ belongs to the trajectory of the point $x$. If $V\subset M^4$ is an  $f^t$-invariant subset and  each trajectory of $f^t|_{V}$ intersects a cross-section  $\Sigma\subset V$, then  $\Sigma$ is called a global cross-section for $f^t|_{V}$.
\end{dfn}

Let $f^t\in G_{1,2,k}$, $k\in \{4,5\}$, $\sigma_1, \sigma_2$ be the saddle equilibria of the flow $f^t$ of Morse indices 1, 2, respectively, $M^4$ be the ambient manifold, and $\varphi: M^4\to [0,4]$ be the energy function of the flow $f^t$ satisfying the conclusion of Proposition~\ref{SM}. For an arbitrary $c_2\in (1,2)$, we set $$ \Sigma=\varphi^{-1}(c_2), \Lambda_s=\Sigma\cap W^s_{\sigma_1}, \lambda_u=\Sigma\cap W^u_{\sigma_2}.$$

\begin{dfn}\label{sh-exdef} The set $S_{f^t}=\{\Sigma, \Lambda_s, \lambda_u\}$ 
is called the scheme of the flow  $f^t\in G_{1,2,k}$.
\end{dfn} 

\begin{dfn}\label{sh-eq} Schemes $S_{f^t}=\{\Sigma, \Lambda_s, \lambda_u\}$, $S_{{f'}^t}=\{\Sigma', \Lambda'_s, \lambda'_u\}$ of  flows $f^t, {f'}^t\in G_{1,2,k}$ are called equivalent if there exists  a homeomorphism $h: \Sigma\to \Sigma'$ such that $h(\Lambda_s)=\Lambda_s'$, $h(\lambda_u)=\lambda_u'$. 
\end{dfn}

Let us emphasize that the manifolds  $\lambda_u, \Lambda_s$ and $\Sigma$   are not equipped with any orientation, and we do not require in Definition~\ref{sh-eq} that  $h$ is  orientation-preserving. We  remark that  $\Sigma$ is a global cross-section for the restriction $f^t|_{M^4\setminus ({\rm cl}\, W^u_{\sigma_1}\cup {\rm cl}\, W^s_{\sigma_2})}$ of $f^t$ to the set $M^4\setminus ({\rm cl}\, W^u_{\sigma_1}\cup {\rm cl}\, W^s_{\sigma_2})$.

\begin{prp}\label{indep} The equivalence class  of the scheme does not depend on the choice of an energy function $\varphi$ and a cross-section  $\Sigma$. Namely, if $\widehat \Sigma$ is   a global cross-section for $f^t_{M^4\setminus ({\rm cl}\, W^u_{\sigma_1}\cup {\rm cl}\, W^s_{\sigma_2})}$ different from $\Sigma$, and $\hat \Lambda_s=\widehat \Sigma\cap W^s_{\sigma_1}, \hat \lambda_u=\widehat \Sigma\cap W^u_{\sigma_2}$, then the scheme $\{\widehat{\Sigma}, \hat \Lambda_s, \hat \lambda_u\}$ is equivalent to the scheme $\{\Sigma, \Lambda_s, \lambda_u\}$.
\end{prp}
\begin{proofn} The homeomorphism $h$, which realizes the equivalence of schemes, associates with each point $x\in \widehat \Sigma$ a point $h(x)$ lying in the intersection of the orbit $\mathcal O_x$ of $x$ and the cross-section $\Sigma$.
\end{proofn}

By Proposition~\ref{hd}, the manifold $\Sigma$ is diffeomorphic either to the sphere $\mathbb S^3$ or to the direct product $\mathbb S^3\times \mathbb S^1$.
Recall that a smooth  knot    $\lambda\subset \mathbb S^3$ is called a \textit{trivial knot} if there exists a smooth  embedding $e: \mathbb B^2\to \mathbb S^3$ such that $e(\partial \mathbb B^2)=\lambda$.  We refer to a smoothly embedded  simple  closed curve $\lambda \subset \mathbb S^2\times \mathbb S^1$ as a \textit{trivial knot in $\mathbb S^2\times \mathbb S^1$} if there  exists a homeomorphism $h: \mathbb S^2\times \mathbb S^1\to \mathbb S^2\times \mathbb S^1$ such that $h(\lambda)=\{x\}\times \mathbb S^1$, $x\in \mathbb S^2$.  We remark that  the trivial knots in $\mathbb S^3$ and $\mathbb S^2\times \mathbb S^1$ share one common property: the complement to  an open tubular neighborhood of the knot in the ambient manifold $\mathbb S^3$ or $\mathbb S^2\times \mathbb S^1$  is homeomorphic  to a solid torus $\mathbb B^2\times \mathbb S^1$.

\begin{lem}\label{seq1} Let $f^t\in G_{1,2,k}$. Then $\lambda_u$ is a trivial knot smoothly embedded in $\Sigma$, and $\Lambda_s$ is a 2-sphere smoothly embedded in $\Sigma$. If $k=4$, then $\Lambda_s$ does not divide $\Sigma$.
\end{lem}
\begin{proofn} By definition, $\Lambda_s=\Sigma\cap W^s_{\sigma_1}$ is a closed two-dimensional manifold that is the global cross-section for the restriction of $f^t$ to $W^s_{\sigma_1}\setminus \sigma_1$. Let us show that   $\Lambda_s$ is a strong deformation retract of the set $W^s_{\sigma_1}\setminus \sigma_1$.   Let  $r:W^s_{\sigma_1}\setminus \sigma_1\to \Lambda_s $ be a map   that associates with any point $x\in W^s_{\sigma_1}\setminus \sigma_1$ a point $r(x)=f^{t_x}(x)\in \Lambda_s=\Lambda_s\cap \mathcal O_x$ of the intersection of $\Lambda_s$ with the trajectory $\mathcal O_x$ of $x$. The map $r$ is a retraction    that is homotopic to the identity map $id: W^s_{\sigma_1}\setminus {\sigma_1}\to  W^s_{\sigma_1}\setminus {\sigma_1}$ by means of a homotopy  $$h_\tau(x)=f^{\tau t_x}(x), \tau\in [0,1].$$ The homotopy $h_\tau$ keeps all points of $\Lambda_s$ fixed. Then  $\Lambda_s$ is a strong deformation retract $X^u_{\sigma}$. As a consequence,  $\Lambda_s$  is homotopically  equivalent to  $W^s_{\sigma_1}\setminus \sigma_1$.

Since the equilibrium state of $\sigma_1$ is hyperbolic, $W^s_{\sigma_1}\setminus \sigma_1$ is homeomorphic to $\mathbb R^3\setminus \{O\}$, and, consequently,  is homotopically equivalent to the sphere $\mathbb S^2$. Hence,   $\Lambda_s$  is simply connected.  Due to classical theorem on classification of closed two-dimensional manifolds up to diffeomorphisms  (see, for instance~\cite[Chapter~9]{Hi}),  $\Lambda_s$ is diffeomorphic to $\mathbb S^2$. 

Let  us show that in the  case $k=4$,   $\Lambda_s$ does not divide $\Sigma$.  We recall that $A_{f^t}=W^u_{\sigma_1}\cup \omega$ is a circle  topologically embedded  in $M^4$, and $M^4$ is homeomorphic to $\mathbb S^4$. Then we  may determine a {\it linking number ${\rm lk}(\Lambda_s, A_{f^t})$} of $\Lambda_s, A_{f^t}$  as an algebraic  number (taking  orientations into account) of points of the intersection $B^3\cap A_{f^t}$, where   $B^3\subset M^4$ is a topologically embedded  three-dimensional ball  (see~\cite[\S 8 Ch.\,2]{Pr-hom}). Suppose that $\Lambda_s$ divides   $\Sigma$. Since $\Sigma$  is diffeomorphic to $\mathbb S^2\times \mathbb S^1$,  then  $\Lambda_s$ bounds a three-dimensional ball $B^3\subset \Sigma$ (see~\cite[Lemma~2]{Miln}). Since $\Sigma\cap A_{f^t}=\emptyset$, then  $B^3\cap A_{f^t}=\emptyset$ and    ${\rm lk}(\Lambda_s, A_{f^t})=0$. From other hand, since $\Lambda_s$ is a two-dimensional sphere in $W^s_{\omega_1}$ which is diffeomorphic to $\mathbb R^3$.   Schoenflies Theorem (see Statement~\ref{sch} in Section~\ref{tm} below) implies  that $\Lambda_s$ is a boundary of  a ball  $B^3_{\sigma_1}= W^s_{\sigma_1}\cap M_{c_2}$. Since $B^3_{\sigma_1}\cap A_{f^t}=\sigma_1$, then ${\rm lk}(\Lambda_s, A_{f^t})=\pm 1$,   that contradicts to the previous observation. Hence, $\Lambda_s$  does not  divide $\Sigma$. 

Now let us  prove that the knot $\lambda_u$ is trivial. Recall that $M_{c_2}=\varphi^{-1}([0,c_2])$, $c_2\in (1,2)$. Set $N_{c_2}=M^4\setminus {\rm int}\,M_{c_2}$.
The function $\psi=4-\varphi$ is the energy function for the flow $f^{-t}$ and $N_{c_2}=\psi^{-1}([0,4-c_2])$. From the properties of the Morse function, Proposition~\ref{SM} and Theorem~\ref{top}, it follows that $N_{c_2}$ is obtained by gluing a 2-handle to the ball $N_{3.5}=\psi^{-1}([0,0.5])$. Therefore, $\partial N_{c_2}$ is obtained from the sphere $\partial N_{3.5}$ by an  integer Dehn  surgery along the knot $\lambda=\partial N_{3.5}\cap W^s_{\sigma_2}$. But $\partial N_{c_2}=\Sigma$, and $\Sigma$, as proved above, is diffeomorphic either to $\mathbb S^2\times \mathbb S^1$ or to $\mathbb S^3$. It is proven in~\cite[Theorems~3,2]{GoLu} (see also Statement~\ref{GoLu-a} in Section~\ref{tm} below) that nontrivial Dehn surgery along nontrivial knot never yields the   sphere or  $\mathbb S^2\times S^1$. Then, the knot $\lambda$ is trivial in $\partial N_{3.5}$. The manifolds $\partial N_{3.5}\setminus \lambda$ and $\Sigma\setminus \lambda_u$ are diffeomorphic by means of a  diffeomorphism that assigns to a point $x\in \partial N_{3.5}\setminus \lambda$ the  point $\mathcal O_x\cap (\Sigma\setminus \lambda_u)$. By [\cite[Theorem~1]{GoLu}) (see Statement~\ref{GoLu-2a} in Section~\ref{tm}), knots on  the sphere $\mathbb S^3$ are ambient homeomorphic if  their complements are homeomorphic.  For the case $k=5$ this immediately implies that  the knot $\lambda_u$ is trivial.

Consider the case $k=4$. In this case, the knot $\lambda$ belongs to the sphere $\partial N_{3.5}$, and the knot $\lambda_u$ belongs to the manifold $\Sigma$, which is diffeomorphic to $\mathbb S^2\times \mathbb S^1$. Let $N_{\lambda_u}\subset \Sigma$ be a compact tubular neighborhood of the knot $\lambda_u$. By what was proved above, the manifold $C=\Sigma\setminus {\rm int}\,N_{\lambda_u}$ is diffeomorphic to the complement in the sphere $\partial N_{3.5}$ to an open tubular neighborhood of the trivial knot $\lambda$. Hence, $C$ is diffeomorphic to a solid torus $\mathbb B^2\times \mathbb S^1$.  Recall that the {\it a  meridian} of a solid torus $C$ is a simple closed curve $\mu\in \partial C$ that is non-homotopic to zero in $\partial C$ and bounds a disk in $C$.  

Up to an isotopy class of a gluing map and an  orientation of a meridian, there  is only one way to obtain $\mathbb S^2\times \mathbb S^1$ by gluing two solid tori, and in this way  the   gluing  map  sends  a  meridian of the first solid torus to a meridian of the second one. %Therefore, isotopic classes of  meridians of the solid tori $P, Q$ coincide. 
Let  $D\in \mathbb S^2$ be a disk containing the point $x\in \mathbb S^2$, inside, $N_0=D\times \mathbb S^1$, $C_0=(\mathbb S^2\setminus {\rm int}\, D)\times \mathbb S^1$, and $h_0: N_{\lambda_u}\to N_0$ be a homeomorphism such that $h_0(\lambda_u)=x\times \mathbb S^1$. Since $h_0$ is a homeomorphism of solid tori, it maps a meridian of  $N_{\lambda_u}$ to a meridian  of $N_0$ which is also a meridian of $C_0$.  It follows that the homeomorphism $h_0|_{\partial N_{\lambda_u}}=h_0|_{\partial C}$ extends to a homeomorphism $h_1: C\to C_0$ such that $h_1|_{\partial C}=h_0|_{\partial C}$ (see, for example,~\cite[Ch.2, \S 5, Ex.5]{Ro}). Then the homeomorphism $h: \Sigma_{f^t}\to \mathbb S^2\times \mathbb S^1$, defined by the formula
$$h(x)=\begin{cases}h_0(x), &{\rm if}\, x\in N_{\lambda_u}\cr h_1(x),  &{\rm if}\, x\in C, \end{cases}$$ maps the knot $\lambda_u$ to the knot $\{x\}\times \mathbb S^1$. Therefore, the knot $\lambda_u$ is trivial. The lemma is proven.
\end{proofn}

Note that in case $k=5$, the  Jordan-Brown Theorem and the Generalized Schoenflies Theorem  (see Statement~\ref{sch} in Section~\ref{tm} below) imply that  $\Lambda_s$ divides  $\Sigma$ into two connected components, the closure of each of which is a three-dimensional ball.

Let $\lambda, \Lambda\subset \Sigma$ be a smoothly embedded 1-and 2-spheres that  intersect each other transversely. We fix the orientation of all three manifolds. If, at a point $x\in \lambda\cap \Lambda$, the orienting tangent frames to $\lambda, \Lambda$ generate a tangent space to $\Sigma$ whose orientation coincides with the chosen orientation of $\Sigma$, then we assign $+1$ to the point $x$. Otherwise, we assign $-1$ to the point $x$. The sum of all such numbers for all points in $\lambda\cap \Lambda$ is called the {\it intersection index} of the pair $\lambda, \Lambda$. The intersection index is a homological invariant. Therefore, it follows directly from Lemma~\ref{seq1} that the intersection index of the knot $\lambda_u$ and the sphere $\Lambda_s$ is equal to one for $k=4$ and  to zero for $k=5$. The trajectory of any point $x\in \lambda_u\cap \Lambda_s$ is the heteroclinic trajectory lying in the intersection $W^u_{\sigma_2}\cap W^s_{\sigma_1}$. Thus, the following statement is proven.

\begin{cor}\label{nmb_of_het} If $f^t\in G_{1,2,k}$ and $k=4\ (5)$, then the flow $f^t$ has an odd $($resp., even$)$ number of heteroclinic trajectories.
\end{cor}

\begin{figure}
\centering{\includegraphics[width=0.4\textwidth]{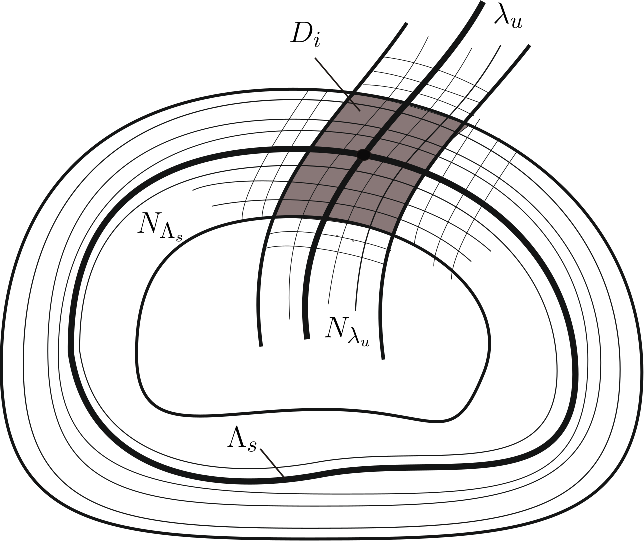}}
\caption{A component $D_i$ of the intersection of  canonical tubular neighborhoods $N_{\Lambda_s},  N_{\lambda_u}$.}
\label{tr}
\end{figure}

We remark that if  $\Lambda_s\cap \lambda_u\neq \emptyset$, then, due to  classical results about  transversal intersections (see, for instance, \cite[Ch.4, Section 6, Exercise~2]{Hi}),  there  exist tubular neighborhoods $N_{\Lambda_s}, N_{\lambda_u}\subset \Sigma$ of  $\Lambda_s$,  $\lambda_u$, respectively, such that the intersection $N_{\Lambda_s}\cap N_{\lambda_u}$ consists of a finite set of connected components $D_1, \dots, D_{b}$ each of which is homeomorphic to $\mathbb B^2\times [0,1]$, and the fibers of the two-dimensional and one-dimensional foliations on this component,  defining the direct product structure, lie on the fibers of foliations transverse to $\lambda_u, \Lambda_s$ that are  defined by the structure of tubular neighborhoods $N_{\Lambda_s}, N_{\lambda_u}$ (see Fig.~\ref{tr}). We will refer to such neighborhoods as \textit{canonical tubular neighborhoods} of $\Lambda_s, \lambda_u$ in $\Sigma$. If $\Lambda_s\cap \lambda_u=\emptyset$, we refer to arbitrary  disjoint compact  tubular neighborhoods of $\Lambda_s, \lambda_u$ as canonical.

\begin{cor}\label{inters1} If each of flows $f^t, {f'}^t\in G_{1,2,4}$  has exactly one heteroclinic trajectory, then the  schemes of  $f^t, {f'}^t$  are equivalent.
\end{cor}
\begin{proofn} Since $\Lambda_s$ and $\lambda_u$ intersect  each other transversely at a unique point, there exist canonical  tubular neighborhoods $N_{\Lambda_s}, N_{\lambda_u}\subset \Sigma$ such that the intersection $N_{\Lambda_s}\cap N_{\lambda_u}$  is homeomorphic to $\mathbb B^2\times [0,1]$.

Since the characteristic cross-section $\Sigma$ is diffeomorphic to $\mathbb S^2\times \mathbb S^1$, there exists a covering $\pi:\mathbb R^3\setminus \{O\}\to \Sigma$. Since the sphere $\Lambda_s$ does not divide $\Sigma$, its complete preimage $\pi^{-1}(\Lambda_s)$ is a countable set of pairwise disjoint smoothly embedded two-dimensional spheres, each of which bounds a ball in $\mathbb R^3$ whose interior contains the point $O$. Let $A\subset \mathbb R^3$ be the compact subset bounded by the pair of spheres $S_0, S_{1}\subset p^{-1}(\Lambda_s)$ such that ${\rm int}\, A\cap \pi^{-1}(\Lambda_s)=\emptyset$. Set $$\widehat N_{\Lambda_s}=p^{-1}(N_{\Lambda_s})\cap A,\  \widehat N_{\lambda_u}=p^{-1}(N_{\lambda_u})\cap A,$$ $$\widehat D=A \setminus {\rm int}(\widehat N_{\Lambda_s}\cup \widehat N_{\lambda_u}).$$ The boundary of $\widehat D$ is a piecewise smooth sphere $S^2$. The Generalized Schoenflies Theorem (see Statement~\ref{sch}) implies that $\widehat D$ is homeomorphic to the ball $\mathbb B^3$. We set $D=\pi(\widehat D)$. Since the restriction of $\pi$ to ${\rm int A}$ is a homeomorphism, $D$ is homeomorphic to the ball.

We denote by $N_{{\Lambda'_s}}, N_{\lambda'_u}, D'\subset \Sigma'$ the sets corresponding to the flow ${f'}^t$ similar  to the sets $N_{{\Lambda_s}}, N_{\lambda_u}, D\subset \Sigma$. Then there exists a homeomorphism $$\xi: N_{{\Lambda_s}}\cup  N_{\lambda_u}\to N_{{\Lambda'_s}}\cup  N_{\lambda'_u}$$  such that
$\xi(\Lambda_s)=\Lambda'_s, \xi(\lambda_u)=\lambda'_u$. By Alexander Theorem (see Statement~\ref{Al}), any homeomorphism from the boundary of a ball can be extended to the ball. Hence,   $\xi$ can be extended to a homeomorphism $\Xi:D\to D'$ such that $\Xi|_{\partial D}=\xi|_{\partial D}$. Therefore, there is a homeomorphism $h: \Sigma\to \Sigma'$ such that $h(\Lambda_s)=\Lambda_s'$, $h(\lambda_u)=\lambda_u'$, which means that the  schemes of $f^t, {f'}^t$ are equivalent.
\end{proofn}

One more corollary of Lemma~\ref{seq1}  can be proved similarly to~\cite[Theorem~4]{MeZh13}. We recall that a closed  manifold $N^l$ of dimension $l\geq 1$ topologically embedded in a manifold $M^n$ is {\it locally flat at a point $x\in N^l$}  if there is a neighborhood $U_x\in M^n$ and a homeomorphism $h:U_x\to \mathbb R^n$ that sends $U\cap N^l$ to an affine subspace of dimension $l$.  A manifold  $N^l\subset M^n$ is called \textit{locally flat} if it is  locally flat at any point $x\in N^l$. We remark that locally flat manifold $N^l\subset M^n$ is a topological (but not necessarily   smooth) submanifold of $M^n$.  A manifold $N^l$ which is not locally flat at    least at one point $x$ is called \textit{ wild}.  

\begin{cor}\label{flat} The closure ${\rm cl}\, W^s_{\sigma_2}$ of the set $W^s_{\sigma_2}$ is a locally flat 2-sphere in $M^4$.
\end{cor}

We note that in~\cite{MeZh13, GuSa24},  examples of gradient-like  flows with two saddle equilibria, the closures of whose invariant manifolds are wild two-dimensional spheres, are constructed.

\section{NECESSARY AND SUFFICIENT CONDITIONS FOR TOPOLOGICAL EQUIVALENCE OF FLOWS FROM $G_{1,2,k}$}

In this section, we prove that the scheme of the flow $f^t\in G_{1,2,k}$ introduced in Definition~\ref{sh-exdef} is a complete topological invariant, that is, the following theorem holds.

\begin{teo}\label{cond} Flows $f^t, {f'}^t\in G_{1,2,k},\  k\in \{4,5\}$, are topologically equivalent if and only if their schemes are equivalent. 
\end{teo}

Theorem~\ref{clas} is an immediate consequence of Theorem~\ref{cond} and Corollaries~\ref{nmb_of_het},~\ref{inters1}. Theorem~\ref{clas+} follows from Theorem~\ref{cond}, Lemma~\ref{cp2} which is proved in Section~\ref{tm}, and Lemma~\ref{rea}, the proof of which (independent of the proof of Theorem~\ref{cond}) is presented in  Section~\ref{r}.

The necessity of the conditions of Theorem~\ref{cond} follows directly from the definition of topological equivalence and Proposition~\ref{indep}. The proof of sufficiency is presented below in the sequence of Propositions~\ref{Step1}-\ref{Step4}.

\subsection{Refining the properties of the homeomorphism $h: \Sigma\to \Sigma'$ satisfying the definition of the equivalence of the schemes}

Let  $N_{{\Lambda_s}}, N_{\lambda_u}\subset \Sigma$ be canonical tubular neighborhoods of  $\Lambda_s$,  $\lambda_u$, respectively.  Since  $\Sigma$ is an orientable manifold, the neighborhood $N_{\lambda_u}$  is a solid torus. Let  $\tilde \lambda_u\subset \partial N_{\lambda_u}$  be a  longitude  (a simple closed curve that is not null-homotopic and   intersects some  meridian of $N_{\lambda_u}$ at exactly  one point)   whose intersection with each connected component of the set $N_{{\Lambda_s}}\cap N_{\lambda_u}$ consists of a single one-dimensional arc  which is a fiber of the one-dimensional foliation of $N_{{\Lambda_s}}$. %For a while, we do not fix any  orientation of $\tilde \lambda_u$, as well as of knot $\lambda_u$.  

Let $N_{{\Lambda'_s}}, N_{\lambda'_u}$, $\widetilde \lambda'_u$ be similar manifolds for a flow ${f'}^t$ whose scheme is equivalent to the scheme of the flow $f^t$.

\begin{prp}\label{Step0} Let the schemes of the flows $f^t, {f'}^t\in G_{1,2,k}$ be equivalent. Then there is a homeomorphism $\eta: \Sigma\to \Sigma'$ such that:
\begin{enumerate}
\item $\eta(\Lambda_s)=\Lambda'_s$, $\eta(\lambda_u)=\lambda'_u$,
\item $\eta(N_{{\Lambda_s}})=N_{_{\Lambda'_s}}$, $\eta(N_{\lambda_u})=N_{\lambda'_u}$;
\item if $k=4$, then  $\eta(\widetilde \lambda_u)=\widetilde \lambda'_u$.
\end{enumerate}
\end{prp}
\begin{proofn} We prove the proposition for  $k=4$ (in case $k=5$ arguments are similar). In this case, by Proposition~\ref{hd} and Lemma~\ref{seq1}, the manifold $\Sigma$ is diffeomorphic to $\mathbb S^2\times \mathbb S^1$, and the sphere $\Lambda_s$ does not divide $\Sigma$.

By the condition, the schemes of the flows  $f^t, {f'}^t\in G_{1,2,4}$ are equivalent,  that  means the  existing of  a homeomorphism $h: \Sigma\to \Sigma'$ such that $h(\Lambda_s)=\Lambda_s'$, $h(\lambda_u)=\lambda_u'$ (see Definition~\ref{sh-eq}). We show that the existence of $h$ implies the existence of the homeomorphism $\eta$.

Set $C=\Sigma\setminus {\rm int}\, N_{\Lambda_s}, C'=\Sigma'\setminus {\rm int}\, N'_{\Lambda_s}$. The proof consists of two steps. In the first step, we show that there exists a  homeomorphism $\eta_{_{C,C'}}: C\to C'$ such that $\eta_{_{C,C'}}(\lambda_u\cap C)=\lambda'_u\cap C'$,
$\eta_{_{C,C'}}(N_{\lambda_u}\cap C)=N_{\lambda'_u}\cap C'$, and
$\eta_{_{C,C'}}(\widetilde \lambda_u\cap C)=\widetilde \lambda'_u\cap C'$. In the second step, we show that the homeomorphism $\eta_{_{C,C'}}$ can be extended on the set $N_{\Lambda_s}$  to  the desired homeomorphism $\eta$.

{\it Step 1. Constructing the homeomorphism $\eta_{_{C,C'}}: C\to C'$}

Consider the covering $\pi:\mathbb R^3\setminus \{O\}\to \Sigma$.
Since the sphere $\Lambda_s$ does not divide $\Sigma$, its complete preimage $\pi^{-1}(\Lambda_s)$ is a countable set of pairwise disjoint smooth  two-dimensional spheres  each of which bounds a ball in $\mathbb R^3$ whose interior contains the point $O$. Let $A\subset \mathbb R^3$ be the compact subset bounded by the pair of spheres $S_0, S_{1}\subset p^{-1}(\Lambda_s)$ such that ${\rm int}\, A\cap \pi^{-1}(\Lambda_s)=\emptyset$. Set $\widehat N_{\Lambda_s}=\pi^{-1}(N_{\Lambda_s})\cap A$. Note that one of the connected components of $\pi^{-1}(C)$ lies in ${\rm int}\, A$ and the restriction of $\pi$ to this component is a diffeomorphism onto $C$. We will identify this connected component with $C$ denoting it by the same symbol, and keep the notation $\lambda_u\cap C$, etc. for sets $\pi^{-1}(\lambda_u)\cap C$, etc. Note that the  Annulus Theorem (see Statement~\ref{an}) implies that the sets $A, C$ and each connected component of the set $\widehat N_{\Lambda_s}$  are diffeomorphic to the annulus $\mathbb S^2\times [0,1]$. Denote by $N_{c}$ the collar of the boundary $\partial C$ in $C$, that is, the image of a smooth embedding $e: \mathbb S^2\times [0,1]\times \mathbb S^0\to C$ such that $ e(\mathbb S^2\times \{0\}\times \mathbb S^0)=\partial C$. Set $N=\widehat N_{\Lambda_s}\cup N_c$ and denote by $N_{-1}, N_{1}$ the connected components of the set $N$. We will define an arbitrary point $p\in N_{\pm 1}$ by two coordinates $x\in \partial C, y\in [0,1]$  and assume that $(x,0)\subset \partial A, \widehat N_{\Lambda_s}\cap N_{\pm 1}=\{(x,y)|\, y\in [0,1/2]\}$.

We define a diffeomorphism $\psi: A\to C$ by setting
\begin{equation}
\psi(p)=\begin{cases}p, & {\rm if}\,  p\in A\setminus {\rm int}\, N;\cr
(x, 0.5y+0.5), & {\rm if}\, p=(x,y)\in N_{\pm 1}
\end{cases}
\end{equation}

Let $\pi',\psi', A', C'$ denote  maps and sets similar to those introduced above, but corresponding to the flow ${f'}^t$. The homeomorphism $h: \Sigma\to \Sigma'$ defines a homeomorphism $$\hat h: A\to A'$$ such that $\pi'\hat h|_A=h \pi|_{A}$. We define the homeomorphism $g_{_{C,C'}}: C\to C'$ setting  $$g_{_{C,C'}}=\psi\hat h\psi^{-1}.$$ By construction, $g_{_{C,C'}}(\lambda_u\cap C)=\lambda'_u\cap C'$, $g_{_{C,C'}}(N_c)= N'_c$.

Due to  results of Bing and Moise, see~\cite[Theorem 8]{Bi}, a compact locally flat subset of a closed 3-manifold is a subpolyhedron, so the compact arcs $\lambda'_u\cap C'$ and their compact neighborhoods $N_{\lambda'_u}\cap C',\  g_{_{C,C'}}(N_{\lambda_u}\cap C)$ are subpolyhedra of the annulus  $C'$. It follows from the uniqueness of a  polyhedral neighborhood (see~\cite[Theorem 4.11, Ch.4]{RuSa})  that there exists a homeomorphism $h_{_{C', C'}}: C'\to C'$ such that $$h_{_{C',C'}}(g_{_{C,C'}}(N_{\lambda_u}\cap C))=N_{\lambda'_u}\cap C',\  h_{_{C',C'}}(\lambda'_u\cap C')=\lambda'_u\cap C'.$$ Due to  Lemma~\ref{dehn}  there exists a homeomorphism $\tilde h_{_{C', C'}}: C'\to C'$ such that $$\tilde h_{_{C',C'}}(g_{_{C,C'}}(N_{\lambda_u}\cap C))=N_{\lambda'_u}\cap C',\  \tilde h_{_{C',C'}}(\lambda'_u\cap C')=\lambda'_u\cap C'$$ and $$\tilde h_{_{C', C'}}(g_{_{C,C'}}(\tilde \lambda_u\cap C))=\tilde \lambda'_u\cap C.$$ Then the superposition $$\eta_{_{C,C'}}=\tilde h_{_{C', C'}}g_{_{C', C'}}$$ is the desired homeomorphism.

{\it Step 2. Extending the homeomorphism $\eta_{_{C,C'}}$ to the set $N_{\Lambda_s}$}.

Recall that from the definition of neighborhoods $N_{\Lambda_s}, N_{\lambda_u}$ it follows that there exists a homeomorphism $$\xi: \mathbb S^2\times [-1,1]\to N_{\Lambda_s}$$ such that $\xi(\mathbb S^2\times\{0\})=\Lambda_s$, and for each connected component $D$ of the set $N_{\Lambda_s}\cap N_{\lambda_u}$ there is a ball $B^2\subset\mathbb S^2$ such that $D=\xi(B^2\times [-1,1])$, while $\xi(O\times [-1,1])\subset \lambda_u$ for some interior point $O$ of the ball $B^2$. Moreover,  $\xi(\{K\}\times [-1,1])\subset \widetilde \lambda_u$ for some point $K\in \partial B^2$. Denote by $\xi': \mathbb S^2\times [-1,1]\to N_{\Lambda'_s}$ a homeomorphism with similar properties.

We define the homeomorphism $$\theta: \mathbb S^2\times \{-1,1\}\to \mathbb S^2\times \{-1,1\}$$ by $$\theta={\xi'}^{-1}\eta_{_{C,C'}}\xi|_{\mathbb S^2\times \{-1,1\}}.$$ Since each of the mapps $\xi, \xi', \eta_{_{C,C'}}$ is a homeomorphism of the annulus, the homeomorphism $\theta$ either simultaneously preserve or simultaneously reverse the orientation of both spheres $\mathbb S^2\times \{-1\}, \mathbb S^2\times \{1\}$. Then Statement~\ref{s-maps} implies that  there exists an isotopy $\theta_t:\mathbb S^2\to \mathbb S^2$ such that $\theta_0=\theta|_{\mathbb S^2\times \{-1\}}, \theta_1=\mathbb S^2\times \{1\}$. We define the homeomorphism $\Theta: \mathbb S^2\times [-1,1]\to \mathbb S^2\times [-1,1]$ by setting for $x\in \mathbb S^2, y\in [-1,1]$

$$\Theta(x,y)=(\theta_{(y+1)/2}(x),y).$$

Now the desired homeomorphism $\eta: \Sigma\to \Sigma'$ is defined by the formula

\begin{equation}
\eta(p)=\begin{cases}\eta_{_{C,C'}}(p), &{\rm if}\, p\in C;\cr
\xi'(\Theta(\xi^{-1}(p))), &{\rm if}\, p\in N_{\Lambda_s}.\end{cases}
\end{equation}

\end{proofn}

\subsection{Construction of consistent canonical neighborhoods of saddle equilibria}

In this section, we prove  Proposition~\ref{Step1} establishing that a handle decomposition of  the ambient manifold  can be adjusted in such a way that  the foliation which is  determined  by the structure of the direct product  $\mathbb B^i\times \mathbb B^{4-i}$ on each handle  became  $f^t$-invariant in the sense of  the following definition.

\begin{figure}
\centering{\includegraphics[width=0.4\textwidth]{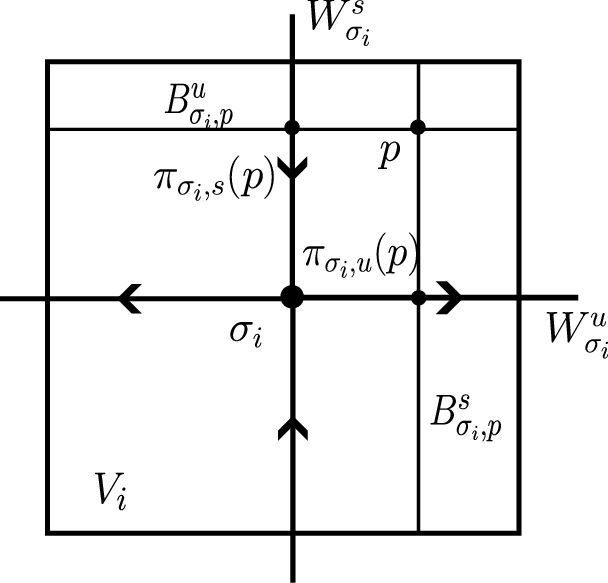}}
\caption{Canonical neighborhood $V_i$  of a saddle  equilibrium $\sigma_i$}
\label{agr}
\end{figure}

\begin{dfn}\label{can} Let $\sigma_i$ be a saddle equilibrium state of Morse index $i$.  A compact neighborhood $V_i$ of $\sigma_i$ is called a {\it canonical compact neighborhood} if
there exist continuous maps $\pi_{\sigma_i,u}: V_{i}\to W^u_{\sigma_i}, \pi_{\sigma_i,s}: V_{i}\to W^s_{\sigma_i}$ with the following properties {\rm(}see Fig.~\ref{agr}\rm{)}.

\begin{enumerate}
\item for any points $x\in V_{i}\cap W^u_{\sigma_i}, y\in V_{i}\cap W^s_{\sigma_i}$ the sets $B^s_{\sigma_i,x}=\pi_{\sigma_i,u}^{-1}(x), B^u_{\sigma_i,y}=\pi_{\sigma_i,s}^{-1}(y)$ are smoothly embedded compact balls of dimension $4-i, i$, respectively that intersect each other  transversely;
\item $B^s_{\sigma_i,\sigma_i}\subset W^s_{\sigma_i}$, $B^u_{\sigma_i,\sigma_i}\subset W^u_{\sigma_i}$;
\item $f^{t}(B^s_{\sigma_i,x})\subset B^s_{\sigma_i,f^{t}(x)}$, $f^{t}(B^u_{\sigma_i,y})\supset B^u_{\sigma_i,f^{t}(y)}$ for any $t>0$,  $x\in V_{i}\cap W^u_{\sigma_i}, y\in V_{i}\cap W^s_{\sigma_i}$.
\end{enumerate}
\end{dfn}

The maps $\pi_{\sigma_i,u}, \pi_{\sigma_i,s}$ define a direct product structure $B^{u}_{\sigma_i, \sigma_i}\times B^s_{\sigma_i, \sigma_i}$ in a neighborhood of $V_{i}$, that is, each neighborhood of $V_i$ is an $i$-handle. As before, we will denote by $P_i$ the foot of the handle $V_i$ and by $Q_i$ the closure of the complement to  $P_i$ in $\partial V_i$. The foliations $\{B^s_{\sigma_i,x}\}$, $\{B^u_{\sigma_i,y}\}$ will be refereed as {\it stable} and {\it unstable} foliations of $V_{\sigma_i}$. We prove in Proposition~\ref{consist} that $V_1, V_2$ can be chosen in such a way that $Q_1, P_2$  are canonical tubular neighborhoods of $\Lambda_s, \lambda_u$, respectively,  in  $\Sigma$.

\begin{dfn} \label{consist} Let $f^t\in G_{1,2,k}$. Canonical neighborhoods $V_1, V_2$ of saddle equilibria $\sigma_1, \sigma_2\in \Omega_{f^t}$ are called consistent if either $V_1\cap V_2=\emptyset$ or $V_1\cap V_2=Q_1\cap P_2\subset \Sigma$ and   consists of a finite set $D=\{D_1, \dots, D_k\}$ of three-dimensional balls and for any point  $p\in V_1\cap V_2$ and the ball $D_i\subset D$ such that $p\in D_i$,    the following inclusion holds:

    $$B^u_{\sigma_1,p}\cap D_i=B^u_{\sigma_2, p}\cap D_i,  B^s_{\sigma_2, p}\cap D_i\subset B^s_{\sigma_1,p}\cap D_i.$$
\end{dfn}

\begin{prp}\label{Step1} For any flow $f^t\in G_{1,2,k}$, there exists a handle decomposition  of the ambient  manifold $M^4$ such that each handle contains a unique equilibrium state of  $f^t$ and the index of any  handle  is equal to the Morse index of the corresponding equilibrium state. Moreover,
the handles of indices 1 and 2 are the consistent canonical neighborhoods of saddle equilibria.
\end{prp}
\begin{proofn}
We prove the proposition for the case $k=4$ (the proof for $k=5$ is similar).
Due to  Propositions~\ref{SM},~\ref{hd},  the manifold $M^4$ admits a handle decomposition $H_0\cup H_1\cup H_2\cup H_4$ such that each handle $H_i$ is a compact neighborhood of an equilibrium state of the flow $f^t$ of the corresponding index. Moreover,  $W^u_{\sigma_i}\cap \partial H_i$, $W^s_{\sigma_i}\cap \partial H_i$ are  foot  and belt spheres of the handle $H_i$, $i\in \{1,2\}$.  By definition, $\Sigma=\partial (H_0\cup H_1)$. 

Let  $e_i: \mathbb B^i\times \mathbb B^{n-i}\to H_i$ be a  diffeomorphism defining the handle structure. For a point $p\in H_i$, we  denote by $\widetilde B^s_{i,p} ( \widetilde B^u_{i,p})$ the leaf of the foliation $e_i(\{x\}\times \mathbb B^{4-i})$  (resp., $e_i(\mathbb B^{i}\times \{y\})$) passing through $p$. Note that $\widetilde B^u_{i,\sigma_i}=W^u_{\sigma_i}\cap H_i, \widetilde B^s_{i,\sigma_i}=W^s_{\sigma_i}\cap H_i$.  

First, we construct a compact neighborhoods $V_{1}$  of $\sigma_1$ enriched by the stable $f^t$-invariant foliation. Let   $\varepsilon_{0}>0$ be a number such that for any smooth  3-ball $B^s_p$ which is $\varepsilon_0$-close  (in the $C^1$ topology) to $\widetilde B^s_{1,\sigma_1}$, and for any point $q\in \widetilde B^s_{1,\sigma_1}$  the intersections $B^s_p\cap \widetilde B^u_{1,q}$ is transverse.  We denote by  $p_+, p_-$  the  points of the intersection $W^u_{\sigma_1}\cap \partial H_0$. Due to classical $\lambda$-lemma (see~\cite{PaMe}, Lemma 7.1 and the remark in $\S 7$ of Chapter 2),  there exists $T(\varepsilon_1)>0$ such that for all $t>T(\varepsilon_0)$ the connected component $B^s_{1,f^{-t}(p_{\pm})}$ of the set $f^{-t}(\widetilde B^s_{1,p_\pm})\cap H_1$ which   contains the point $f^{-t}(p_\pm)$,  is a 3-ball  which is $\varepsilon_0$-close to $\widetilde B^s_{1,\sigma_1}$. Then the union $$V_{1}=\left(\cup_{t\geq T(\varepsilon_0)+1}  B^s_{1,f^{-t}(p_\pm)}\right)\cup \widetilde B^s_{1,\sigma_1}$$ forms a compact neighborhood of the point $\sigma_1$, foliated by smoothly embedded 3-balls transversal  to $\widetilde B^u_{1,\sigma_1}$,  and this foliation satisfies to  condition  3 of  Definition~\ref{can}.  
By construction,  $\partial \widetilde B^s_{1,f^{-t}p_\pm}\subset \Sigma$.
Set $Q_{1}=\partial V_{1}\cap \Sigma$. The boundaries of the balls $B^s_{1,f^{-t}(p_{\pm})}$ form a two-dimensional foliation of the set $Q_{1}$, and  we  will refer to  these foliation as a stable foliation of $Q_{1}$. We remark that $V_0=H_0\cup {\rm cl}(H_1\setminus  V_1)$ is homeomorphic to a ball and  $\partial (V_0\cup V_1)=\Sigma$, so we may  consider $V_0$ and $V_1$ as handles of indices $0$ and $1$, respectively,  in a handle decomposition of $M^4$. 

Similarly to the previous paragraph, we construct a compact neighborhood  $V_{2}$ of $\sigma_2$ enriched by the unstable $f^t$-invariant foliation. For this,  we remark that the set  $Q_2=H_2\cap \partial H_4$ is a tubular neighborhood of a knot $k=W^s_{\sigma_2}\cap \partial H_4$ and  there exists $T>0$ such that for any disk $\widetilde B^u_{2,p}$, $p\in k$, the connected component $B^u_{2,f^{t}(p)}$ of the set $f^{t}(\widetilde B^u_{2,p})\cap H_2$ which   contains the point $f^{t}(p)$ is sufficiently close to  $\widetilde B^u_{2,\sigma_2}$ to assume that the union  
$$V_2=\left(\cup_{p\in k, t\geq T+1}B^u_{2,f^t(p)}\right)\cup \widetilde B^u_{2,\sigma_2}$$ form a compact neighborhood  of $\sigma_2$ with the following properties: the set $P_{2}=V_{2}\cap \Sigma$ is a compact  neighborhood of $\lambda_u$ in $\Sigma$ enriched by an unstable foliation formed by   $\partial B^u_{2,f^t(p)}=B^u_{2,f^t(p)}\cap \Sigma$, and each fiber $\partial B^u_{2,f^t(p)}$  is transverse  to the stable foliation of  $Q_{1}$.   

 We extend the unstable  one-dimensional   foliation defined on $Q_1\cap P_{2}$ to an one-dimensional foliation on the whole  $Q_1$,  transverse to the the stable foliation of $Q_1$. Let $I^u_x$ denote the fiber of this one-dimensional foliation passing through the point $x\in \Lambda_s$.

Let $\varepsilon_{1}>0$ be such that for any smooth compact arc $B^u_p$, $\varepsilon_1$-close to $\widetilde B^u_{1,\sigma_1}$, the intersection of  $B^u_p$ with   each fiber  of the stable foliation of $V_1$  is transverse. Applying  $\lambda$-lemma once more, we find $T(\varepsilon_1)>0$ such that for any point $x\in \Lambda_s$ and for all $t>T(\varepsilon_1)$, the connected component $B^u_{1,f^t(x)}$ of the set $f^{t}(I^u_{x})\cap V_{1}$ containing the point $f^{t}(x)$, is an arc which is  $\varepsilon_1$-close to $B^u_{\sigma_1, \sigma_1}=W^u_{\sigma_1}\cap V_1$. Then the union  $$\{ B^u_{1,f^t(x)}, x\in \Lambda_s, t\geq T(\varepsilon_1)+1\}\cup B^u_{\sigma_1, \sigma_1}$$ form an   unstable $f^t$-invariant foliation of $V_{1}$. Similarly, we define a stable $f^t$-invariant foliation for $V_2$. The union $N=V_0\cup V_{1}\cup V_{2}$ is a manifold with boundary $\partial N\subset W^u_\alpha\setminus \alpha$. Moreover, $\partial N$ is a  global cross-section for the flow $f^t|_{W^u_{\alpha}\setminus \alpha}$. Using arguments of the proof of Lemma~\ref{seq1}, we prove that $\partial N$ is homotopy equivalent to $\mathbb S^3$. Due to three-dimensional Poincare Conjecture proved by G.~Perelman (see~\cite{PC}),  $\partial N$ is a three-dimensional sphere. By Generalized Schoenflies Theorem,  $\partial N$ is a boundary of a four-dimensional ball  $V_4\subset W^u_\alpha$. We choose this ball as a handle of index $4$. Then the union $V_0\cup V_{1}\cup V_{2}\cup V_4$ is the desired handle decomposition of the manifold $M^4$.

%forms a compact neighborhood of $\sigma_1$, foliated by smoothly embedded 1-balls transversal to $\widetilde B^s_{1,\sigma_1}$,  $\{B^s_{1,f^{-t}(x)}\}$, and this foliation satisfies the Definition~\ref{can}. Set $$V_{\sigma_1}=V_{1,\varepsilon_0}\cap V_{1,\varepsilon_1}, V_{\omega}=H_0\cup V_{1,\varepsilon_0}\setminus {\int V_{\sigma_1}$$%, choose $T>0$ such that $f^T(V)\cap H_0\subset (\widetilde B^s_+\cup \widetilde B^s_-)$, 
%and set $V_{\sigma_1}=f^T(V)\cap H_1$. Since any iteration of a canonical neighborhood is also a canonical neighborhood, and the intersection of $V_{\sigma_1}$ with $\partial H_0, \Sigma$ consists of leaves belonging to the constructed $f^t$-invariant foliations, $V_{\sigma_1}$ is a canonical neighborhood of the saddle $\sigma_1$.

 %$V_{2,\varepsilon_4}$ of the saddle equilibrium state $\sigma_2$ that is transverse to the foliation $\{B^u_{2,\sigma_2}\}$, and define a canonical neighborhood $V_{\sigma_2}$ that is consistent with $V_{\sigma_1}$ and has a non-empty intersection with $\Sigma$ and $\partial H_0$. 
\end{proofn}

\subsection{Local topological conjugacy of flows $f^t, {f'}^t\in G_{1,2,k}$ in consistent canonical neighborhoods of saddle equilibria}

In this section,  we consider flows  $f^t, {f'}^t\in G_{1,2,k}$ whose schemes   are equivalent by means of  the homeomorphism $h: \Sigma\to \Sigma'$ and construct in Proposition~\ref{local_conj} a homeomorphism that conjugates  the restrictions of  flows $f^t, {f'}^t$ on the  union of their  consistent canonical neighborhoods  coherently with  $h$.

We recall that,  by the definition of consistent neighborhoods,
$\lambda_u=\partial B^u_{\sigma_2,\sigma_2},\Lambda_s=\partial B^s_{\sigma_1, \sigma_1}$, and $P_2=\Sigma\cap V_{\sigma_2}, Q_1=\Sigma\cap V_{\sigma_1}$ are tubular neighborhoods of the knot $\lambda_u$ and the sphere $\Lambda_s$ in $\Sigma$, respectively. Similar relations are true for the consistent neighborhoods $V'_1, V'_2$ of the saddle equilibrium states of the flow ${f'}^t$, and $h(\Lambda_s)=\Lambda'_s), h(\lambda_u)=\lambda'_u$. We will suppose that that $h(P_2)=P'_2=\Sigma'\cap V'_2$ and $h(Q_1)=Q'_1=\Sigma'\cap V'_1$.  Due to Proposition~\ref{Step0}, this assumption does not reduce generality.

\begin{prp}\label{local_conj}
Let $\{V_1,V_2\}, \{V'_1, V'_2\}$ be pairs of consistent canonical neighborhoods of saddle equilibria of flows $f^t, {f'}^t\in G_{1,2,k}$, whose schemes   are equivalent by means of  the homeomorphism $h: \Sigma\to \Sigma'$ such that  $h(Q_1)=Q'_1, h(P_2)=(P'_2)$. Then there exist  homeomorphisms $\chi_{1}: V_{1}\to V'_{1}, \chi_{2}: V_{2}\to V'_{2}$ with the following properties:

\begin{enumerate}
\item $\chi_{i}f^t={f'}^t\chi_{i}$ for all $t$ for which the right-hand and left-hand sides of the equality are defined;
\item $\chi_2(p)=\chi_1(p)$ for all points $p\in V_1\cap V_2$;
\item $\chi_i(\Lambda_s\cap \lambda_u)=h(\Lambda_s\cap \lambda_u)$;
\item the restrictions $\chi_2h^{-1}|_{\partial P'_2}: \partial P'_2\to  \partial P'_2$, 
$\chi_1h^{-1}|_{\Lambda'_s}: \Lambda'_s\to  \Lambda'_s$ are  isotopic to identity.  
\end{enumerate}
\end{prp}
\begin{proofn}

We describe the construction of the homeomorphisms $\chi_1, \chi_2$ for the case  when $V_1\cap V_2\neq \emptyset$ (in the case $V_1\cap V_2=\emptyset$, the construction is similar).

Let us recall that the boundary of $V_2$ is a three-dimensional sphere that is represented  as a union of two solid tori $P_2\subset \Sigma, Q_2$ with a common boundary $\partial B^u_{\sigma_2, \sigma_2}\times \partial B^s_{\sigma_2, \sigma_2}=P_2\cap Q_2$ and cores $\lambda_u, \partial B^s_{\sigma_2}$, respectively. Let $(x,y)\in \partial B^u_{\sigma_2, \sigma_2}\times \partial B^s_{\sigma_2, \sigma_2}$ be a pair of points.  Then $\mu_x=\{x\}\times \partial B^s_{\sigma_2, \sigma_2}$, $\mu_y=\partial B^u_{\sigma_2, \sigma_2}\times \{y\}$ are meridians of $P_2, Q_2$, and longitudes of $Q_2, P_2$, respectively. Hence, the homology classes of  knots $\mu_x, \mu_y$ (endowed with some orientations) in  $\partial Q_2$  may be considered as generators of the  group  $H_1(\partial Q_2, \mathbb Z)$. 

We   fix    orientations on the knots $\lambda_u$, $\partial B^s_{\sigma_2,\sigma_2}$   and assume that the  pair  $(\mu_x, \partial B^s_{\sigma_2,\sigma_2})$, as well as the  pair $(\mu_y, \lambda_u)$, is  oriented co-directionally in $\partial V_2$.  For  the sake of definiteness, we may identify $V_2$ with a  subset of $\mathbb R^4$    assuming that $\sigma_2=\{O\}$, $\mathbb B^u_{\sigma_2, \sigma_2}\subset Ox_1x_2,\mathbb B^s_{\sigma_2, \sigma_2} \subset Ox_3x_4$ and choose the  orientation of the  knot  $\lambda_u\, (\partial B^s_{\sigma_2, \sigma_2})$ with respect to the  direction of the shortest path from  $Ox_1$ to $Ox_2$ (resp., from the $Ox_3$ to $Ox_4$).

  We chose  orientations of   knots in a canonical neighborhood  $V'_2$ of the saddle $\sigma'_2$ of the  flow ${f'}^t$  and fix knots $\mu'_x, \mu'_y\subset \partial Q'_2$ whose homology  classes generate the group $H_1(\partial Q'_2,
  \mathbb Z)$ in the same way as it was done above  for $H_1(\partial Q_2,\mathbb Z)$.  The consistent canonical neighborhoods, projections, and other objects for  the flow ${f'}^t$, similar to those constructed above for the flow $f^t$, will be denoted by $V'_{i}$, $\pi_{\sigma'_i,u}: V'_{i}\to W^u_{\sigma'_i}, \pi_{\sigma'_i,s}: V'_{i}\to W^s_{\sigma'_i}$, etc.

%Let us remark that the orientation of $\lambda_u, \mu_x$ determines the orientation of $\Lambda_u$ and $\Sigma$. Indeed, let  $p\in \Lambda_u\cap \lambda_s$.  By definition~\ref{consist},  $B^u_{\sigma_2, p}\subset P_2\cap \Lambda_u$, and we orient $B^u_{\sigma_2}$ in such a way that  $[\partial  B^u_{\sigma_2, p}]=[\mu_x]$ in $H_1(\partial P_2, \mathbb Z)$.  Then, the  orientation of  $B^u_{\sigma_2, p}$ determine an  orientation of $\Sigma$. The orientations of $B^u_{\sigma_2, p}$ and  $\lambda_u$ determines a triple of vectors in the tangent space to $\Sigma$ at the point $p$, that  determines an orientation of $\Sigma$.

Since $h|_{P_2}: P_2\to P'_2$ is a homeomorphism of solid tori, it maps  a meridian  of $P_2$ to a  meridian of $P'_2$ (disregarding orientation, see, for instance,~\cite[Chapter 2, E]{Ro}). Hence,  $[h(\mu_x)]=\pm [\mu'_x]$ in homology group  $H_1(P'_2, \mathbb Z)$.  We may assume without loss of generality that $[h(\mu_y)]=\pm [\mu'_y]$.    For $k=4$ this fact immediately follows from  Proposition~\ref{Step0}. For $k=5$,  according to Lemma~\ref{seq1}, $\Sigma$ is a three-dimensional sphere, and $\lambda_u$ is  a trivial knot. It means that $\Sigma\setminus {\rm int Q_2}$ is a solid torus, and $\mu_y$ can be considered as a meridian of this solid torus (see, for instance,~\cite[Ch.2, E; Ch.\,9]{Ro}). Applying again the fact that a homeomorphism of the solid torus maps a meridian to a meridian (disregarding orientation), we get $[h(\mu_y)]=\pm [\mu'_y]$. 

According to Definition~\ref{can}, any point $z\in V_i$ may be described by the pair $(x, y)$ of coordinates $x=\pi_{\sigma_i,u}(z)\subset B^{u}_{\sigma_i, \sigma_i}$ and $y=\pi_{\sigma_i,s}(z)\subset B^{s}_{\sigma_i, \sigma_i}$, and so do $z'\in V'_i$.  Hence, to construct the desired  homeomorphism $\chi_i: V_i\to V'_i$, its enough to define  appropriate  homeomorphisms  $$h^u_i: B^{u}_{\sigma_i, \sigma_i}\to B^{u}_{\sigma'_i, \sigma'_i},\  h^s_i: B^{s}_{\sigma_i, \sigma_i}\to B^{s}_{\sigma'_i, \sigma'_i}$$ and set $$\chi_i(z)=(h^u_i(x), h^s_i(y)).$$

 Let   
$$g^u_1:\partial B^u_{\sigma_1,\sigma_1}\to \partial {B'}^u_{\sigma'_1, \sigma'_1}, \  g^s_2:\partial B^s_{\sigma_2,\sigma_2}\to \partial {B'}^s_{\sigma'_2, \sigma'_2}$$ be arbitrary homeomorphisms  that acts on the orientations as follows. Homeomorphism $g^u_1$ (resp., $g^s_2$) is   orientation  preserving in case $[h(\mu_y)]=[\mu'_y]$ (resp., $[h(\mu_x)]=[\mu'_x]$)  and orientation reversing otherwise.    

We associate a time $t_{x} (t_y)\in \mathbb R$ to any  point $x\in B^u_{\sigma_1, \sigma_1}$ $(y\in B^s_{\sigma_2,\sigma_2})$  such that $f^{t_{x}}(x)\in \partial B^u_{\sigma_1,\sigma_1}\  (f^{t_{y}}(y)\in \partial B^s_{\sigma_2,\sigma_2}$) and define homeomorphisms  

$$h^u_1: B^u_{\sigma_1,\sigma_1}\to B^u_{\sigma'_1,\sigma'_1},\ h^s_2: B^s_{\sigma_2,\sigma_2}\to B^s_{\sigma'_2,\sigma'_2}$$ by

$$h^u_1(x)=\begin{cases}
{f'}^{-t_{x}}(g^u_1(f^{t_{x}}(x))), & {\rm if}\, x\in B^u_{\sigma_1,\sigma_1}\setminus \sigma_1 \cr
\sigma'_1, & {\rm if}\, x=\sigma_1;
\end{cases}$$

$$h^s_2(y)=\begin{cases}
{f'}^{-t_{y}}(g^s_2(f^{t_{y}}(y))), & {\rm if}\, y\in B^s_{\sigma_2,\sigma_2}\setminus \sigma_2\cr
\sigma'_2, & {\rm if}\, x=\sigma_2;
\end{cases}$$

Let $p\in \lambda_u\cap \Lambda_s$. Then $B^u_{\sigma_1, p} \subset \lambda_u=\partial B^u_{\sigma_2,\sigma_2}, B^s_{\sigma_2, p}\subset \Lambda_s=\partial B^s_{\sigma_1,\sigma_1}$. For each point $p\in \lambda_u\cap \Lambda_s$ we define homeomorphisms

$$\xi^u_p: B^u_{\sigma_1, p}\to B^u_{\sigma'_1, h(p)},\  \xi^s_p: B^s_{\sigma_2, p}\to B^s_{\sigma'_2, h(p)}$$ such that

$$h^u_1\pi_{u,\sigma_1}=\pi_{u,\sigma'_1}\xi^u_p,\ h^s_2\pi_{s,\sigma_2}=\pi_{s,\sigma'_2}\xi^s_p.$$

It follows  from the equivalence of schemes of the flows $f^t, {f'}^t$ that the homeomorphism $\xi^u_p$  extends to a homeomorphism

$$g^u_2:\partial B^u_{\sigma_2,\sigma_2}\to \partial {B'}^u_{\sigma'_2, \sigma'_2}$$

which  is orientation preserving if $[h(\mu_y)]=[\mu'_y]$ and orientation reversing otherwise, and  such that 

$$g^u_2|_{B^u_{\sigma_1, p}}=\xi^u_p|_{B^u_{\sigma_1, p}}.$$

We remark that  $B^s_{\sigma'_2, h(p)}$ is a two-dimensional disk embedded in the two-dimensional sphere $\Lambda'_u$ and $\xi^s_ph^{-1}|_{B^s_{\sigma'_2, h(p)}}: B^s_{\sigma'_2, h(p)}\to B^s_{\sigma'_2, h(p)}$ is  the orientation preserving  homeomorphism. Then, using  a collar $C_p\subset  \Sigma'\setminus {\rm int B^s_{\sigma'_2, h(p)}}$ of $\partial B^s_{\sigma'_2, h(p)}$ in $\Sigma'$ and Statement~\ref{s-maps}, we define a   homeomorphism $\psi_p: \Lambda'_u\to \Lambda'_u$ such that  $\psi_p|_{B^s_{\sigma'_2, h(p)}}=\xi^s_ph^{-1}|_{B^s_{\sigma'_2, h(p)}}$ and $\psi_p|_{\Sigma'\setminus (B^s_{\sigma'_2, h(p)}\cup C_p)}={\rm id}|_{\Sigma'\setminus (B^s_{\sigma'_2, h(p)}\cup C_p)}$. It is clear that $C_p$ can be chosen in such a way that $C_p\cap B^s_{\sigma'_2, h(q)}=\emptyset$ for any point $q\in \lambda_u\cap \Lambda_s$, different from $p$. Then the composition  $\Psi: \Sigma'\to \Sigma'$  of homeomorphisms $\psi_p$ defined  for all points $p\in  \lambda_u\cap \Lambda_s$ maps each disk ${B^s_{\sigma'_2, h(p)}}$  to itself and formula $g^s_1=\Psi h|_{\Lambda_u}$ defines  homeomorphism 

$$g^s_1:\partial B^s_{\sigma_1,\sigma_1}\to \partial {B'}^s_{\sigma'_1, \sigma'_1},$$ such that $
g^s_1|_{B^s_{\sigma_2, p}}=\xi^s_p|_{B^s_{\sigma_2, p}}$  for any point $p\in \lambda_u\cap \Lambda_s$. 

We remark that $\partial {B'}^s_{\sigma'_1, \sigma'_1}=\Lambda'_u$ and by construction  homeomorphism $g^s_1h^{-1}|_{\Lambda'_u}$ is orientation preserving, hence, is isotopic to identity. 

We  extend homeomorphisms $g^u_2, g^s_1$ to homeomorphisms
$h^u_2: B^u_{\sigma_2,\sigma_2}\to {B'}^u_{\sigma'_2, \sigma'_2},\  h^s_1: B^s_{\sigma_1,\sigma_1}\to {B'}^s_{\sigma'_1, \sigma'_1}$ similar to the construction of homeomorphisms $h^u_1, h^s_2$.

Finally, we define the desired homeomorphisms $\chi_{i}: V_i\to V'_i, i\in \{1,2\},$ by the relation
$$\chi_{i}(z)=\pi_{\sigma'_i,u}^{-1}(h^u(\pi_{\sigma_i,u}(z)))\cap \pi_{\sigma'_i,s}^{-1}(h^s(\pi_{\sigma_i,s}(z))),\ z\in V_{i}.$$

Note that by  construction, curves $\chi_2(\mu_x)=\mu_{h^u_2(x)}, \chi_2(\mu_y)=\mu_{h^s_2(y)}$ are  meridians  of the solid tori $P'_2, Q'_2$, respectively, and $\chi_2h^{-1}|_{\mu_{h^u_2(x)}}, \chi_2h^{-1}|_{\mu_{h^s_2(y)}}$ are orientation preserving maps. Hence, $\chi_2h^{-1}|_{\partial P'_2}$ induce identity map of homology group $H_1(\partial P'_ 2, \mathbb Z)$  of the torus $\partial P'_2$ and is isotopic to identity (see, for instance,~\cite[Ch.2,\, D]{Ro}).  
\end{proofn}

\subsection{The completion of the proof of Theorem~\ref{cond}}

\begin{prp}\label{Step4} Let the  schemes of flows $f^t, {f'}^t\in G_{1,2,k}$ be equivalent. Then there exists a homeomorphism $H: M^4\to M^4$ such that $Hf^t={f'}^tH$.
\end{prp} We prove the proposition for $k=4$ (the arguments for $k=5$ are similar). Due to  Proposition~\ref{Step1},  $M^4$ admits a decomposition into handles in which the handles of indices 1 and 2 are consistent  compact canonical neighborhoods $V_1, V_2\  (V'_1, V'_2)$ of saddle equilibrium states of the flow $f^t{f'}^t$, respectively). Recall that $\partial V_i=P_i\cup Q_i$, where $P_i$ is the foot of the handle $V_i$. Then $Q_1, P_2$ are tubular neighborhoods of the sphere $\Lambda_s$ and the knot $\lambda_u$ in the characteristic cross-section $\Sigma$.

Due to Proposition~\ref{Step0} the schemes of the flows $f^t, {f'}^t$ are equivalent by means of the  homeomorphism $h:\Sigma\to \Sigma'$ such that $h(Q_1)=Q'_1, h(P_2)=P_2'$. Due to  Proposition~\ref{local_conj}, there exist homeomorphisms $\chi_i: V_i\to V'_i$, $i\in \{1,2\}$,  such that $\chi_{i}f^t={f'}^t\chi_{i}$ for all $t$ for which the right-hand and left-hand sides of the equality are defined,  and  the restrictions $\chi_2h^{-1}|_{\partial P'_2}: \partial P'_2\to  \partial P'_2$, 
$\chi_1h^{-1}|_{\Lambda'_s}: \Lambda'_s\to  \Lambda'_s$ are  isotopic to identity. 

Let $N'$ be a collar of $\partial P'_2$ in $\Sigma\setminus {\rm int}\,P'_2$ such that the fibers of the one-dimensional foliation defining the structure $\partial P'_2\times [0,1]$ on $N'$ lie on the two-dimensional spheres foliating the neighborhood $Q'_1$ of $\Lambda'_s$ (the existence of such a collar is guaranteed by the transversality of the intersection of $\Lambda'_s$ with $\lambda'_u$ and~\cite[Ch.4, Section 6, Exercise~2]{Hi}). Using $N'$ and an isotopy to the identity map, we extend $\chi_2h^{-1}|_{\partial P'_2}$ to a map $\Psi:\Sigma'\to \Sigma'$ that is the identity outside $N'\cup P'_2$ and preserves $Q'_1$.

We set $\theta=\Psi h$ and  define a homeomorphism $\eta: Q_{\sigma'_1}\to Q_{\sigma'_1}$ by  $\eta=\chi_{\sigma_1, \sigma'_1}\theta^{-1}|_{Q_{\sigma'_1}}$.   By construction, $\eta|_{P_{\sigma'_2}\cap Q_{\sigma'_1}}=\chi_{\sigma_1, \sigma'_1} \chi^{-1}_{\sigma_2, \sigma'_2}|_{P_{\sigma'_2}\cap Q_{\sigma'_1}}$. Due to Proposition~\ref{local_conj}, the last composition is the identity map. Hence,  $\eta$ is an orientation preserving  homeomorphism of the annulus $Q_{\sigma'_1}$  and therefore is isotopic to the identity map. Using this isotopy and a collar of $Q_{\sigma'_1}$ in $\Sigma_{{f'}^t}\setminus {\rm int}\,Q_{\sigma'_1}$,  we extend $\eta$ to a map $\Psi':\Sigma_{{f'}^t}\to \Sigma_{{f'}^t}$ that is the identity outside  of the  collar and on $P_{\sigma'_2}$.

Finally, we define a homeomorphism $G: \Sigma\to \Sigma'$ by $G=\Psi'\Psi h$.
We remark that $$G|_{Q_1}=\Psi'\Psi h|_{Q_1}=\eta\theta|_{Q_1}=\chi_1\theta^{-1}\theta|_{Q_1}=\chi_1|_{Q_1};$$

$$G|_{P_2}=\Psi h|_{P_2}= \chi_2h^{-1} h|_{P_2}=\chi_2|_{P_2}.$$ 

To each point $z\in S^4\setminus \Omega_{f^t}$ we assign a time $t_z\in \mathbb R$ such that $f^{t_z}(z)\subset \Sigma$ and set

\begin{equation} H(z)=\begin{cases}
{f'}^{-t_z}(G(f^{t_z}(z))), &{\rm if}\, z\in \Sigma_{f^t};\cr
{f'}^{-t_z}(h_{\sigma_i,\sigma'_i}(f^{t_z}(z))), &{\rm if}\, z\in V_{\sigma_i}, i\in \{1,2\}.
\end{cases} \label{bij}
\end{equation}

Formula~(\ref{bij}) defines a continuous bijection on the set $M^4\setminus \Omega_{f^t}$, which has a unique continuous extension to $\Omega_{f^t}$ to the desired homeomorphism.

\section{REALIZATION OF TOPOLOGICAL EQUIVALENCE CLASSES}\label{r}

Let $\Sigma_k$ denote the direct product of $\mathbb S^2\times \mathbb S^1$ for $k=4$ and the sphere $\mathbb S^3$ for $k=5$; $\Lambda, \lambda\subset \Sigma_k$ be a smoothly embedded two-dimensional sphere and a trivial knot, respectively. \begin{dfn} The set $\{\Sigma_k, \Lambda, \lambda\}$  is called  an {\it abstract scheme} if $\Lambda, \lambda$ intersect each other  transversely and in the case $k=4$ the sphere $\Lambda$ does not divide $\Sigma_4$.
\end{dfn}
\begin{lem}\label{rea} For any abstract scheme $\{\Sigma_k, \Lambda, \lambda\}$
there exists a flow $f^t\in G_{1,2,k}, k\in \{4,5\}$, whose scheme is equivalent to the abstract scheme.  
\end{lem}
\begin{proofn} We define a vector field on the handle $H^4_i=\mathbb B^i\times \mathbb B^{4-i}$   by the system $$\begin{cases}\dot{x}=x,\cr \dot y =-y,\end{cases}$$ where $x\in \mathbb B^i, y\in \mathbb B^{4-i}$, $i\in \{0,\dots, 4\}$. In case $k=4\,(5)$, we glue the handle $H^4_1$ to one (or two, respectively) copies of the handle $H^4_0$ using a diffeomorphism that ensures smooth gluing of the vector fields. By smoothing the resulting manifold in the standard way, we obtain a manifold $Y$ whose boundary is diffeomorphic to $\Sigma_k$, and a vector field on it transversal to $\partial Y$ which has two (three) singular points: one (two, respectively) sinks and a saddle of index 1. The stable manifold of the saddle intersects $\partial Y$ along the two-dimensional sphere $\Lambda_s$. In the case $k=4$ this sphere does not divide $ \Sigma_k$. Then there exists a diffeomorphism $h: \partial Y\to \Sigma_k$ such that $h(\Lambda_s)=\Lambda$. We set $\lambda_u=h^{-1}(\lambda)$. We attach the handle $H^4_2$ to $Y$ by means of  a diffeomorphism that maps the foot sphere of the handle  to the knot $\lambda_u$. Since the knot $\lambda_u$ is trivial, the boundary of the resulting manifold $Z$ is diffeomorphic to the sphere $\mathbb S^3$. We attach the handle $H^4_4$ to $Z$, resulting in a closed smooth manifold $M^4$ and a vector field on it which  induces a flow $f^t\in P$ whose scheme is equivalent to the given abstract scheme.
\end{proofn}

\subsection{Proof of Theorem~\ref{realiz}}

The first part of Theorem~\ref{realiz} follows immediately from Lemma~\ref{rea}. We show that for any integer $\gamma \geq 1$, the class $G_{1,2,4}$ contains a countable set of topologically nonequivalent flows with exactly $(2\gamma+1)$ heteroclinic trajectories. To do this, by Lemma~\ref{rea} and Theorem~\ref{cond}, it suffices to provide a countable set of pairwise nonequivalent abstract schemes in each of which the intersection of the knot $\lambda$ and the sphere $\Lambda$ consists of exactly $(2\gamma+1)$ points. Examples of such abstract schemes are shown in Fig.~\ref{diff}. 

\begin{figure}
\centering{\includegraphics[width=0.6\textwidth]{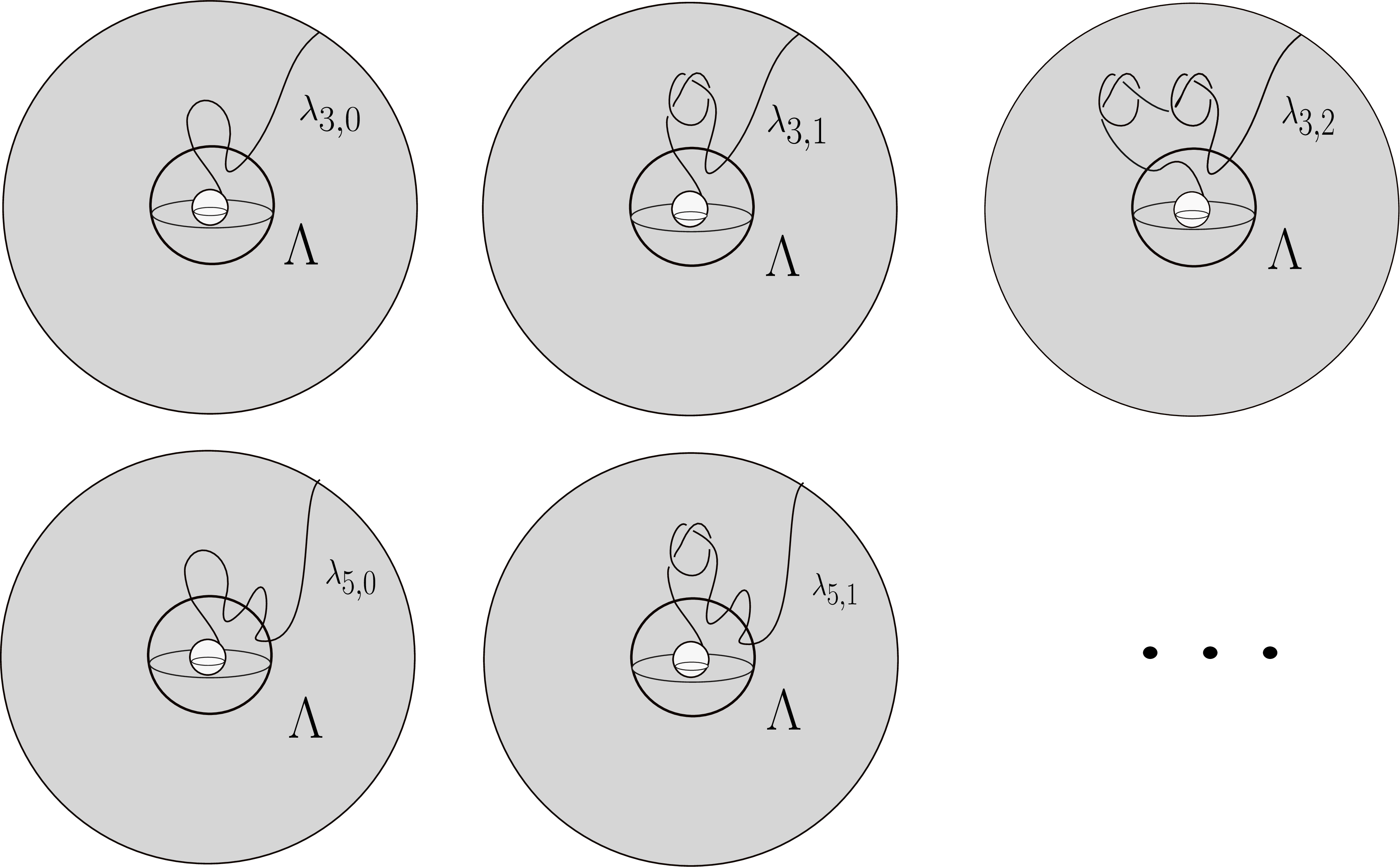}}
\caption{Nontrivial sphere $\Lambda$ and trivial knot $\lambda_{k,i}$ in $\Sigma=\mathbb S^2\times \mathbb S^1$}
\label{diff}
\end{figure}

A  top part of Figure~\ref{diff} shows three cases of mutual arrangement of a homotopically nontrivial two-dimensional sphere $\Lambda$ and trivial knot $\lambda_{3,0}, \lambda_{3,1}, \lambda_{3,2}$ in $\Sigma_4=\mathbb S^2\times \mathbb S^1$, transversely intersecting at three points. To obtain $\mathbb S^2\times \mathbb S^1$, it is necessary to glue together   the boundary spheres of the annulus  highlighted in gray  by a homothety. The series of knots shown in the figure naturally extends to a countable series of trivial knots $\{\lambda_{3, i}\}$.  We recall that, according to a definition of Section~\ref{propsec}, a knot  $\lambda\subset \mathbb S^2\times \mathbb S^1$ is trivial if there exists a homeomorphism $h: \mathbb S^2\times \mathbb S^1\to \mathbb S^2\times \mathbb S^1$ such that $h(\lambda)=\{x\}\times \mathbb S^1$ for some point $x\in \mathbb S^2$. Due to well-known \enquote{light bulb theorem} (see, for instance, \cite[Ex.3 \S F Ch.9]{Ro}), a knot $\lambda\subset \mathbb S^2\times \mathbb S^1$ is trivial if and only if there exists a sphere $\mathbb S^2\times \{y\}$, $y\in \mathbb S^1$, such that the intersection $\lambda\cap (\mathbb S^2\times \{y\})$ is transversal and  consists of a single point. It easy to  verify that for any knot  shown in Fig.~\ref{diff}, such the sphere $\mathbb S^2\times \{y\}$ exists, that proves the triviality of the knot.   We  show that there is no a   homeomorphism $h: \Sigma_4\to \Sigma_4$ such that $h(\Lambda)=\Lambda$, $h(\lambda_{3,i})=\lambda_{3,j}$ for $i\neq j$. Assume the  contrary.
\begin{figure}
\centering{\includegraphics[width=0.54\textwidth]{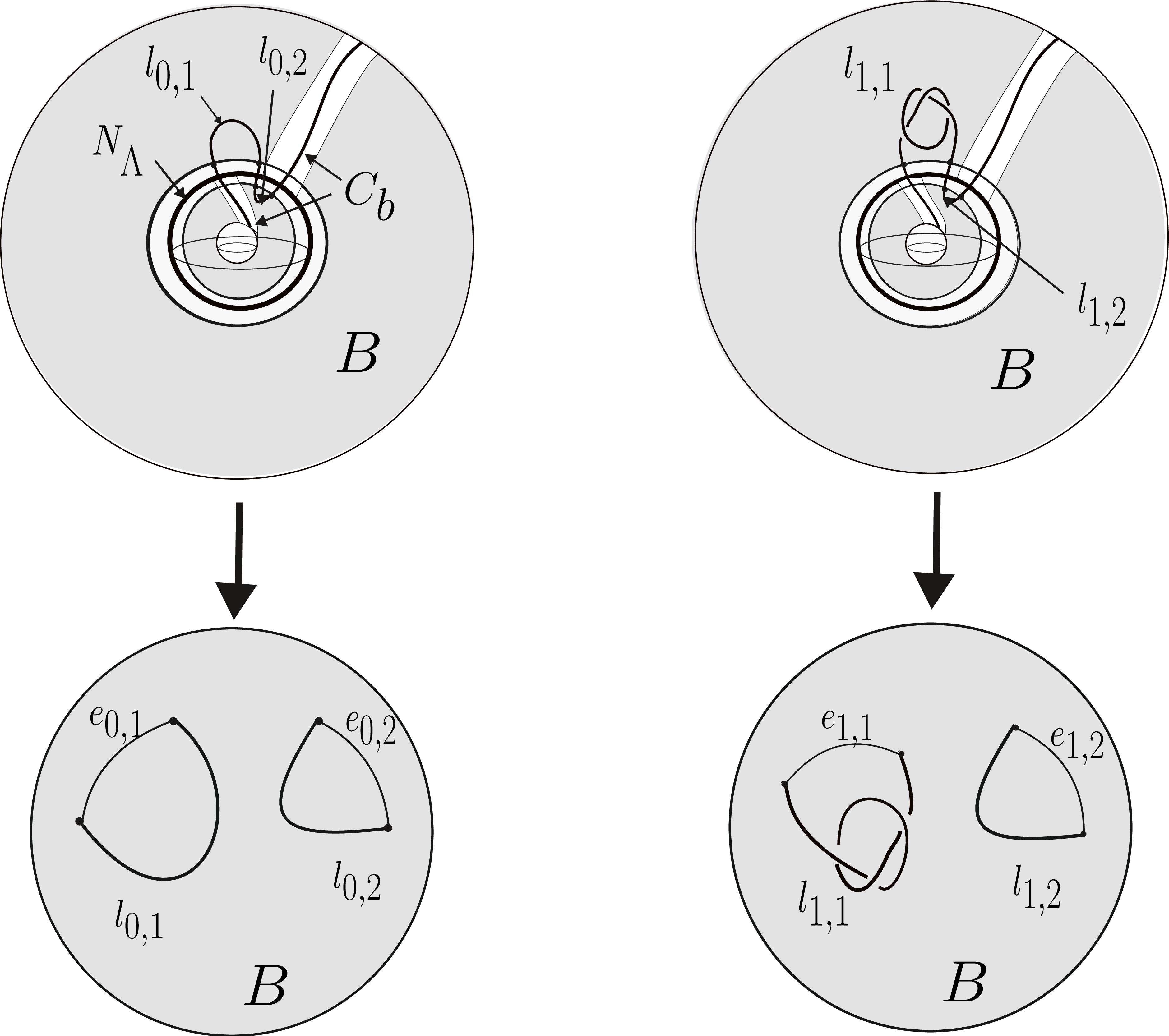}}
\caption{Ball $B=\Sigma\setminus ({\rm int}\, N_{_{\Lambda}}\cup {\rm int}\, C_b)$ and arcs $l_{i,j}, e_{i,j}\subset B$ for $i\in \{0,1\}, j\in \{1,2\}$}
\label{diff+}
\end{figure}

Let $N_{_\Lambda}, N_{\lambda_{3,i}}$ be compact tubular neighborhoods of the sphere $\Lambda$ and the knot $\lambda_{3,i}$ in $\Sigma_4$, respectively. Since the sphere $\Lambda$ does not divide $\Sigma_4$, there exists a compact arc $b_{3,i}\subset \lambda_{3,i}$ whose endpoints lie on the  different connected components of the boundary of the annulus $N_{_\Lambda}$. Let $C_b\subset \lambda_{3,i}$ be a connected component of the set $N_{\lambda_{3,i}}\setminus {\rm int}\,\Lambda$ containing the arc $b_{3,i}$. The set $B=\Sigma_4\setminus {\rm int}\, (N_{_\Lambda}\cup  C_{b})$ is homeomorphic to the ball $\mathbb B^3$. Let $l_{i,1}, l_{i,2}$ denote the connected components of the set $B\cap \lambda_{3,i}$ (see Fig.~\ref{diff+}, which illustrates cases $i\in \{0,1\}$). The endpoints of each of the arcs $l_{i,1}, l_{i,2}$ belong to the boundary $\partial B$ of the ball $B$. We  connect the endpoints of the arcs $l_{i,j}$ by an arc $e_{i,j}\subset \partial B$ such that $e_{i,1}\cap e_{i,2}=\emptyset$ and set  $k_{i,j}=l_{i,j}\cup e_{i,j}$. We glue the ball $\mathbb B^3$ to the ball $B$, arbitrarily identifying the boundary spheres $\mathbb S^2$ and  $\partial B$. As a result, we obtain the sphere $S^3$ and the link $L_{3,i}=\{k_{i,1}, k_{i,2}\}$ in it. By construction, for each $i$, at least one of the knots, say $k_{i,2}$, is trivial, and there exists a three-dimensional ball $D\subset S^3$ such that $k_{i,2}\subset {\rm int}\, D$, $k_{i,1}\cap D=\emptyset$.
The existence of a homeomorphism $h: \Sigma_4\to \Sigma_4$ implies the existence of a homeomorphism $H: S^3\to S^3$ such that $H(k_{i,1})=k_{j,1}$ for $i\neq j$. But the knot $k_{i,1}$ is the connected sum of the trivial knot and $i$ copies of the trefoil. All such knots are pairwise nonequivalent, hence, we get a contradiction.  

The bottom of  Figure~\ref{diff} shows how the series of knots $\lambda_{3,i}$ is modified into a series of knots $\lambda_{2\gamma+1,i}$, each of which has a given number $2\gamma+1$ of intersections with the sphere $\Lambda$. Arguments similar to those presented above prove the existence of a countable set of topologically nonequivalent flows from the class $G_{1,2,4}$ with a given number $2\gamma+1$ of heteroclinic intersections. Thus, Theorem~\ref{realiz} is proved.

%%%%%%%%%%%%%%%%%%%%%%%%%%%%%%

\section{ADDENDUM: TOPOLOGICAL MIXTURE}\label{tm}

This section presents the statements of certain topological results  used in the paper and   auxiliary topological Lemmas~\ref{cp2},~\ref{dehn}, in order of appearing them in the main text.

We recall that $\mathbb B^n$ denote the  unit $n$-ball $\{x\in \mathbb R^n|\  |x|\leq 1\}$, and $\mathbb S^{n-1}$ is the boundary of $\mathbb B^n$. In the paper,  along with the homeomorphic images of $\mathbb B^n$ and $\mathbb S^{n-1}$ and their products, we use one more canonical manifold $\mathbb {CP}^2$. The definition of $\mathbb {CP}^2$ is the following.  We will say that two points $(z_1,z_2,z_3)$ and $(z'_1,z'_2,z'_3)\in \mathbb C^3\setminus \{O\}$  are equivalent, if there exists $k\in \mathbb C\setminus \{O\}$ such that $z'_1=kz_1, z'_2=kz_2, z'_3=kz_3$.   $\mathbb {CP}^2$ is a factor-space of $\mathbb C^3\setminus \{O\}$ by this equivalence relation. Then  $\mathbb {CP}^2$ is an orientable  closed  smooth four-dimensional manifold that admits a Morse function $\varphi: \mathbb {CP}^2\to \mathbb R$ with exactly three critical points: a minimum, a maximum, and a saddle of Morse index 2 (see, for instance,~\cite[\S 19.3]{Pr-comb}). As a consequence, we get that $\mathbb {CP}^2$ is simply connected and   the second homology group  $H_2(\mathbb {CP}^2, \mathbb Z)$ is isomorphic $\mathbb Z$. 

Let $M^4$ be a   simply connected  four-dimensional manifold (smooth or topological). We recall that  an  {\it intersection form} $$Q_{{M^4}}: H_2(M^4, \mathbb Z)\times H_2(M^4, \mathbb Z)\to \mathbb Z,$$ of $M^4$  is a bi-linear unimodular form that associates with each pair of classes $(a,b)\in H_2(M^4, \mathbb Z)$ their homological intersection index (for the  definition of the intersection index  see, for instance,~\cite[\S 5.3]{Pr-hom}). In some basis, the intersection form $Q_{M^4}$ is represented by a non-degenerative unimodular   matrix  $A_{{M^4}}$.  Intersection forms $Q_{{M^4}}, Q_{{N^4}}$ of manifolds $M^4, N^4$ are equivalent if  there exists a non-degenerative unimodular   matrix $C$ with integer elements such that  $A_{{N^4}}=C^{-1} A_{{M^4}}C$. 

Friedman in~\cite{Fr} proved the following result. 

\begin{stat}\label{Fr-a}
 Smooth simply connected closed four-dimensional manifolds $M^4, N^4$ are homeomorphic if and only if their intersection forms are equivalent.      
\end{stat}

As an immediate consequence of Statement~\ref{Fr-a} we get the following facts.  
\begin{cor}[Four-dimensional Poincare Conjecture]\label{Fr-a1}
Smooth simply connected homology 4-sphere  $M^4$ is homeomorphic to $\mathbb S^4$.
\end{cor}
\begin{proofn} By the assumption,  $H_2(M^4, \mathbb Z)$ is trivial. Hence, $Q_{{M^4}}$ is zero. It is also true for $\mathbb S^4$, so, due to Statement~\ref{Fr-a}, $M^4$ and $\mathbb S^4$ are homeomorphic.  
\end{proofn}
\begin{cor}\label{Fr-a2}  If $M^4$  is a simply connected smooth four-dimensional manifold and  $H_2(M^4, \mathbb Z)$ is isomorphic to $\mathbb Z$, then $M^4$ is homeomorphic to $\mathbb {CP}^2$. 
\end{cor}
\begin{proofn} Since $H_2(M^4, \mathbb Z)$ has one generator,  $Q_{{M^4}}$ is determined by a  one-dimensional matrix, that is an integer. Since $Q_{{M^4}}$ is unimodular, this integer equals to $1$ or $-1$. It is also true for $\mathbb{CP}^2$, so, $Q_{{M^4}}$ and $Q_{{\mathbb{CP}^2}}$ are equivalent. Due to Statement~\ref{Fr-a}, $M^4$ and $\mathbb{CP}^2$ are homeomorphic.  
\end{proofn} 

The proof of  the key Lemma~\ref{seq1} is based on Statement~\ref{GoLu-a}  below that uses a concept of  Dehn surgery. The following paragraph is a preamble to the definition of  this.  

Recall that a smooth knot  $\lambda\subset \mathbb S^3$ is an image of a smooth embedding of a circle $\mathbb S^1$. A tubular neighborhood   $\Pi_\lambda$ of a smooth knot $\lambda$ is  a solid torus that is the image of a smooth embedding $e:\mathbb B^2\times \mathbb S^1\to \mathbb S^3$ such that $e(\{O\}\times \mathbb S^1)=\lambda$.   We recall that {\it a meridian}  $\mu_{_\lambda}\subset \partial \Pi_\lambda$  is a smooth knot that is homotopic in $\partial \Pi_\lambda$ to $e(\partial \mathbb B^2\times \{z\})$, where $z\in \mathbb S^1$.  We assume that $\mathbb S^3$ and $\lambda$ are  endowed with some orientation, and orient $\mu_{\lambda}$ according to  the \enquote{gimlet rule}.   {\it A  longitude} $\nu_\lambda\subset \partial \Pi_\lambda$ is an oriented  knot such that the algebraic number of the  points of the intersection $\mu_{_\lambda}\cap \nu_\lambda$ is equal to $1$ or $-1$. We will assume that $\nu_\lambda\subset \Pi_\lambda$ is oriented co-directionally with $\lambda$, so that the algebraic number of the  points of the intersection $\mu_{_\lambda}\cap \nu_\lambda$ is equal to $1$. Then we may determine an integer ${\rm lk}(\lambda, \nu_\lambda)$   that equals to  the  number of revolutions of the knot $\nu_\lambda$ around the meridian $\mu_{_\lambda}$ with respect to the orientation of $\mu_{_\lambda}$. It is well known that ${\rm lk}(\lambda, \nu_\lambda)$  is preserved by an isotopy of the sphere $\mathbb S^3$ and changes the sing while the the orientation of $\lambda$ or $\mathbb S^3$ is changed (see, for instance,~\cite[\S 5, Ch.1]{Pr-hom}).  We will call a longitude $\nu_\lambda$ of $\Pi_\lambda$ {\it canonical} if ${\rm lk}(\lambda, \nu_\lambda)=0$. Considering a  meridian $\mu_\lambda$ and a canonical longitude  $\nu_\lambda$ as generators of   the homology group $H_1(\partial \Pi_\lambda, \mathbb Z)$, one may represent the homology class $[\gamma]$ of any knot $\gamma\subset \partial \Pi_\lambda$  as $[\gamma]=p[\nu_\lambda]+q[\mu_\lambda]=\langle p,q\rangle$, where $p,q$ are coprime integers.

Let $\psi:\partial \mathbb B^2\times \mathbb S^1 \to \partial \Pi_\lambda$ be a homeomorphism  that sends the meridian of $\Pi=\mathbb B^2\times \mathbb S^1$ to a knot on $\partial \Pi_\lambda$, having homology class $\langle p,q\rangle$. The manifold $M^3_{\lambda, \psi}$ obtained from the union  $\Pi\cup (\mathbb S^3\setminus \Int \Pi_\lambda)$    by identifying points  $x\in \partial \mathbb B^2\times \mathbb S^1$ and $\psi(x)$  is said to be obtained from $\mathbb S^3$ {\it by Dehn surgery} along the knot $\lambda$. The Dehn surgery in which the meridian of the solid torus $\Pi$ is glued to the meridian $\mu_\lambda$ is called {\it trivial}, since in this case $M^3_{\lambda, \psi}$ is diffeomorphic to the sphere $\mathbb S^3$. Otherwise, the Dehn surgery is called {\it non-trivial}.
In the case of a non-trivial Dehn surgery along the knot $\lambda$, the ratio $r={q}/{p}$ is called the {\it surgery coefficient}. 

The following fact is proved in~\cite[Theorem 2]{GoLu}. 
\begin{stat}\label{GoLu-a} The manifold $M^3_{\lambda,\psi}$ obtained by a nontrivial Dehn surgery along a nontrivial knot $\lambda$ is not homeomorphic to the sphere $\mathbb S^3$.
\end{stat}

Statement~\ref{GoLu-a}  is the main ingredient of  the proof of the following strong result of \cite[Theorem 1]{GoLu}.   

\begin{stat}\label{GoLu-2a} Let $\lambda, \lambda'\in \mathbb S^3$ be smooth knots such that their complements  $\mathbb S^3\setminus \lambda, \mathbb S^3\setminus \lambda'$  are homeomorphic. Then there is a homeomorphism $h: \mathbb S^3\to \mathbb S^3$ such that $h(\lambda)=\lambda'$. 
\end{stat}

Recall that the Schoenflies Theorem (see \cite{Nu39}, \cite{Moi})
states that any subset of the Euclidean
plane $\mathbb{R}^{2}$ that is homeomorphic to the circle
$\mathbb S^{1}$ is the boundary of a two-dimensional disk in $\mathbb{R}^{2}$. 
According to~\cite{Br60}, this result can be   generalized as follows. 

\begin{stat}[Generalized Schoenflies Theorem] \label{sch} Let $S^{n-1}\subset \mathbb S^{n}$ be a locally flat sphere. Then the closure of every connected component of $\mathbb S^n\setminus S^{n-1}$ is homeomorphic to the ball $\mathbb B^n$.
\end{stat}

\begin{stat}[Annulus Theorem]\label{an}
{ Let $S^{n-1}_0, S^{n-1}_1\subset \mathbb S^{n}$ be disjoint locally  flat spheres. Then the closure in $\mathbb S^{n}$ of the domain bounded by $S^{n-1}_0, S^{n-1}_1$  is homeomorphic to  $\mathbb S^{n-1} \times [-1,1]$.}
\end{stat}

Brown and Gluck in \cite[Theorem 9.4]{BrGl64} (see also \cite[Chapter 5, Theorem 3.2]{Ke66}) have shown   that the Annulus Theorem is closely related to the Stable Homeomorphism Conjecture for the $n$-sphere $S^n$ ($SHC_n$). A homeomorphism $f: S^n \to S^n$ is called {\it stable} if it can be represented  as a finite composition of homeomorphisms, each of which is the identity on some open set. The $SHC_n$ conjecture (now proven; see references below) states that every orientation preserving homeomorphism of  $n$-sphere $S^n$ is stable. Brown and Gluck showed that the $SHC_n$ conjecture implies the Annulus Theorem for dimension $n$, and that the Annulus Theorem for dimensions $k \leq n$ implies the $SHC_n$ conjecture.

For $n \leq 3$, the $SHC_n$ conjecture was proven in \cite{Fi}, which establishes the validity of the Annulus Theorem for dimensions 2 and 3 (in fact, two-dimensional Annulus Theorem  can be  easily derived from the Schoenflies Theorem,  see \cite[Ch. 2, A]{Ro}). For the case $n>4$, the $SHC_n$ conjecture was proven by R. Kirby in 1969 (see \cite{Ki69}). The annulus theorem for the case $n=4$ was proved by F. Quinn in 1984 (see \cite{Qu84}, as well as the exposition of this work in \cite{Ed84}), that also proves $SHC_4$ conjecture.  The following result is one more corollary of $SHC_n$ conjecture  (see~\cite[\S 4]{BrGl64}).

\begin{stat}\label{s-maps}  Homeomorphisms $f, g: \mathbb{S}^n\to \mathbb{S}^n$,  $n\geq 1$ are isotopic if and only if their are either simultaneously  orientation preserving or simultaneously  orientation reversing.  
\end{stat} 

The following classical theorem  sometimes is  called the Alexander trick (see, for instance,~\cite[Proposition~3.22]{RuSa}).

\begin{stat}[Alexander Theorem]\label{Al} For any  homeomorphism $h: \partial \mathbb B^n\to \partial \mathbb B^n$ there exists a homeomorphism $H: \mathbb B^n\to \mathbb B^n$ such that $H|_{\partial \mathbb B^n}=h$. 
\end{stat}
\begin{proofn}
   Let identify a point  $x\in \mathbb B^n$  with its radius-vector. Then the desired homeomorphism  is   $H(x)=||x||*h({x}/{||x||})$.
\end{proofn}

Due to~\cite[Theorems 3.1, 3.2 of Chapter 8]{Hi}, the following statement holds. In fact,  these  theorems were proved under the condition that $e_i, e'_i$ are  smooth embeddings. However, the  proof goes through  for locally flat embeddings,  taking into account the theorem on the homogeneity of topological manifolds (see, for example,~\cite[Lemma 3.33]{RuSa}). 

\begin{stat}\label{balls}
Let $M^n$ be a  closed connected orientable manifold, $X^n \subset M^n$ be a compact submanifold that does not divide $M^n$, and
$e_{i}: \mathbb B^{m}\to M^n\setminus X^n, e'_i: \mathbb B^{m}\to M^n\setminus X^n$, $i\in \{1,k\}$,  be  locally flat  embeddings such that 
$\mathcal B=\bigcup\limits_{i=1}^k e_i(\mathbb B^{m}), \mathcal B'=\bigcup\limits_{i=1}^k e'_i(\mathbb B^{m})$ are disjoint unions of  balls of dimension $m\leq n$. If $m=n$ we assume that  all $e_i, e'_i$ are either simultaneously  orientation preserving or simultaneously  orientation reversing. Then  there exists an isotopy $h_t: M^n\to M^n$ such that
\begin{enumerate}
\item $h_0=\Id_{M^n}$;
\item $h_1e_i=e'_i$ for any $i\in \{1,\dots, k\}$;
\item $h_t|_{X^n}=\Id_{X^n}$ for any $t\in [0,1]$.
\end{enumerate}
\end{stat}

We recall that a {\it handle of index $i$}  is the direct product $H^n_i=\mathbb {B}^i\times \mathbb {B}^{n-i}$. For a  manifold $X$ of dimension $n$  with a non-empty boundary, we denote by   $Y=X\cup_\psi H^{n}_i$ a factor-space $(X\cup H^n_i)/\sim$ of disjoint union of $X$ and $H^n_i$ by minimal equivalence relation $x\sim \psi(x)$, where   $\psi: \partial \mathbb B^i\times \mathbb B^{n-i}\to \partial X$ is  a smooth embedding. We refer to $Y$ as a manifold obtained by gluing a handle $H^n_i$ to $X$. 

\begin{stat}\label{bundl}
Let $Y$ be a manifold with boundary obtained by gluing a handle $H^n_1$ to a ball $\mathbb B^n$. Then $Y$ is homeomorphic either to the direct product $\mathbb B^n\times \mathbb S^1$ or to a non-orientable $\mathbb B^n$-bundle over $\mathbb S^1$.
\end{stat}
\begin{proofn}
   We remark  that the foot $P_1=\partial \mathbb B^1\times \mathbb B^{n-1}$ of the handle $H^n_1$ is a pair of  balls of dimension $(n-1)$. Let $\psi: P_1\to \partial \mathbb B^n$ be a gluing map. It follows from  Statements~\ref{balls},~\ref{Al}, that the  topology of $Y$ does not depend on a choice of the images   $\psi(P_1)\subset \partial \mathbb B^n$ of these balls, and is determined only by isotopy type of $\psi$. Then we may represent $\mathbb B^n$ as $\mathbb B^1\times \mathbb B^{n-1}$ and assume that $Y$ is obtained by gluing two copies of $\mathbb B^1\times \mathbb B^{n-1}$ together by identifying each point  $x\in \{-1,1\}\times \mathbb B^{n-1}$ of one copy  with a point $\psi(x)\in \{-1,1\}\times \mathbb B^{n-1}$ on the other copy. Moreover, we may assume that $\psi$ is either the identity map on both balls $\{-1,1\}\times \mathbb B^{n-1}$,  or the identity on the ball $\{-1\}\times \mathbb B^n$ and coincides with the  reflection $(x_1,x_2,\dots, x_n)\to (-x_1,x_2,\dots, x_n)$ on the  ball $\{+1\}\times \mathbb B^n$.  Hence, $Y$ has a structure of the direct product $\mathbb B^1\times S^1$ in the first case and of a non-orientable $B^n$-bundle over $S^1$ in the second case. 
\end{proofn}

\begin{lem}\label{cp2} Let $\Lambda, \Lambda'\subset \mathbb S^3$ be two-dimensional spheres smoothly embedded in $\mathbb S^3$, and let $\lambda, \lambda'\subset \mathbb S^3$ be smooth trivial knots such that the intersection of the sphere $\Lambda (\Lambda')$ and the knot $\lambda (\lambda')$ is transversal and consists of $b (b')$ points. If $b=b'$, then there exists a homeomorphism $h: \mathbb S^3\to  \mathbb S^3$ such that $h(\Lambda)=\Lambda', h(\lambda)=\lambda'$.
\end{lem}
\begin{proofn} Since the intersection $\lambda\cap \Lambda$ is transversal, $b$ is even. Let $b=0$. Denote by $B_{\lambda}\,  (B_{\lambda'})$ the connected component of $\mathbb S^3\setminus \Lambda\,  (\mathbb S^3\setminus \Lambda')$ that contains the knot $\lambda \, (\lambda')$. Since the knots $\lambda, \lambda'$ are trivial, there exists a homeomorphism $h:  \mathbb S^3\to  \mathbb S^3$ such that $h(\lambda)=\lambda'$. It follows from  Statement~\ref{sch}  that $h(\mathbb S^3\setminus B_\lambda)$ is a three-dimensional  ball. Due to   Statement~\ref{balls},   we can assume that $h(\mathbb S^3\setminus B_\lambda)=\mathbb S^3\setminus B_{\lambda'}$. Thus, for the case $b=0$, the lemma is proved. 

We prove the lemma for the case $b=2k>0$. Let $p_i, \dots, p_{b}$ are points of the intersection $\lambda\cap \Lambda$ ordered with respect to some fixed orientation of $\lambda$. 

Since knot $\lambda$ is trivial, there exists a smooth embedding $e: \mathbb B^2\to \mathbb S^3$ such that $e(\partial \mathbb B^2)=\lambda$. We may assume without loss of generality  that the intersection of $B^2=e(\mathbb B^2)$ and $\Lambda$ is transversal (see, for  instance,  \cite[Theorem~2.2]{Hi}). Then ${\rm int} B^2\cap \Lambda$ consists of two  finite sets $\Gamma_0$, $\Gamma_1$. The set $\Gamma_0$ consists of closed  curves, and the set $\Gamma_1$ consists of open arcs joining  pairs of point  from  $\lambda\cap \Lambda$.  By employing the classical  innermost loop argument and \enquote{cut-and-paste technique} (see, for instance, \cite[Ch.~5, Proof of Theorem~14]{Ro}), we can  eliminate the set $\Gamma_0$ and construct a smoothly embedding in $\mathbb S^3$ disk $D^2$ such that $\partial D^2=\lambda$ and ${\rm int} D^2\cap \Lambda$ consists only of open arcs $l_1, \dots, l_k$. We assume that $p_{2i-1},p_i$ are the  endpoint of the arc $l_i$, $i\in \{1,\dots, k\}$. 

Let $e: \mathbb B^2\times [-1,1]\to \mathbb S^3$ be a smooth embedding such that $e(\mathbb B^2\times \{0\})=D^2$ and $Q=e(\mathbb B^2\times [-1,1])$ intersect $\Lambda$ transversely along pairwise disjoint  two-dimensional disks $\Pi_1, \dots \Pi_k$. We suppose that  $l_i\subset \Pi_i$, $i\in \{1,\dots, k\}$. 

For the pair  $\lambda', \Lambda'$, we define similar  manifolds and denote them  by the same symbols with an added prime.

Let $g: \lambda\to \lambda'$ be a homeomorphism such that $g(b_i)=b'_i$, $i\in \{1,\dots, b\}$. Since the  union of closed arcs $\{{\rm cl}\,l_i\}$ divides $D^2$ into $2k+1$ disks with pairwise disjoint interiors,  there exists a homeomorphism $\eta: D^2\to D^2$ such that $\eta|_{\lambda}=g|_{\lambda}$ and $\eta(l_i)=l'_i$ for any $i\in \{1,\dots, k\}$.  Using the coordinates $x\in \mathbb B^2, y\in [-1,1]$ on the set $Q$ induced by the embedding $e$, we define a homeomorphism
$\psi: Q\to Q$ by $\psi(x,y)=(\eta(x), y)$.

By Statement~\ref{balls}, there exists a homeomorphism $\Psi: \Lambda\to \Lambda'$ such that $\Psi|_{\Lambda\cap Q}=\psi|_{\Lambda\cap Q}$. 

Let   $B_+, B_- (B'_+, B'_-)$  denote the closures  of the connected components of the set $\mathbb S^3\setminus \Lambda
\ (resp.,\, \mathbb S^3\setminus \Lambda')$ such that   $\psi(B_+\cap Q)\subset B'_+$.  Then $R_+=B_+\setminus {\rm Q}, R'_+=B'_+\setminus {\rm Q'}$  are the balls whose boundaries are homeomorphic by means of homeomorphism
$\theta_+: \partial R_+\to \partial R'_+$ such that $$\theta_+(z)=\begin{cases}
\psi(z), & {\rm if}\, z\in \partial Q\cap B_+;\cr
    \Psi(z), & {\rm if}\, z\in \Lambda\setminus Q. 
\end{cases}$$

Due to Statement~\ref{Al}, $\theta_+$ can be extended up to a homeomorphism $\Theta_+:R_+\to R'_+$ such that $\Theta_+|_{\partial R_+}=\theta_+|_{\partial R_+}$. Similarly, we define balls  $R_-=B_-\setminus {\rm Q}, R'_-=B'_-\setminus {\rm Q'}$ and a homeomorphism $\Theta_-:R_-\to R'_-$. Then the desired homeomorphism $h$ is defined by 
$$ h(z)=\begin{cases}
    \Theta_+(z) & {\rm if}\, z\in R_+;\cr
    \Theta_-(z) & {\rm if}\, z\in R_-;\cr
    \psi(z) & {\rm if}\, z\in Q.
\end{cases}$$
\end{proofn}

\begin{lem}\label{dehn}
Let $\mathbb B^2_r=\{x\in \mathbb R^2|\ |x|\leq r\}$, $0<r_1<r<1$; $x\in \partial \mathbb B^2_r$ be an arbitrary  point, and $g: \mathbb B^2_r\times [0,1]\to \mathbb B^2_r\times [0,1]$ be  an orientation preserving  homeomorphism such that $g(\{O\}\times [0,1])=\{O\}\times [0,1]$. Then there exists a homeomorphism $h:\mathbb B^2\times [0,1]\to \mathbb B^2\times [0,1]$ with the following properties:
\begin{enumerate}
\item $h|_{\mathbb B^2_{r_1}\times [0,1]}={\rm id}$;
\item $h(\mathbb B^2_r\times [0,1])=\mathbb B^2_r\times [0,1]$;
\item $h(g(\{x\}\times [0,1]))=\{x\}\times [0,1]$;
\item $h|_{\partial \mathbb B^2\times [0,1]}={\rm id}$.
\end{enumerate}

\end{lem}
\begin{proofn} %Set $l=\{x\}\times [0,1], l'=g(l)$. Let $e:[-1,1]\times [0,1]\to \partial \mathbb B^2_r\times [0,1]$ be a topological embedding such that $e(\{0\}\times [0,1])=l$. Set $N_l=e([-1,1]\times [0,1])$ and denote by $e', N_{l'}$ a similar embedding and its image, which is a compact neighborhood of the arc $l'$ in $\partial \mathbb B^2_r\times [0,1]$.
%The set $$C_l={\rm cl} (\partial \mathbb B_r\times [0,1]\setminus N_l)$$ is homeomorphic to the disk $\mathbb B^2$. By Alexander Theorem (see Statement~\ref{Al}), the restriction of the homeomorphism $e'e^{-1}: N_l\to N_{l'}$ to $\partial N_l$ extends to a homeomorphism $C_l\to C_{l'}$. Therefore, there exists an orientation preserving homeomorphism
%$$\psi: \partial \mathbb B_r\times [0,1]\to \partial \mathbb B^2_r\times [0,1]$$
%such that $\psi(l)=l'$. 

Set $$P_+=(\mathbb B^2\setminus {\rm int}\, \mathbb B^2_{r}) \times [0,1],\  P_-=(\mathbb B^2_{r}\setminus {\rm int}\, \mathbb B^2_{r_1}) \times [0,1].$$

Since $\mathbb B^2\setminus {\rm int}\, \mathbb B^2_{r}$ is homeomorphism to $\mathbb S^1\times [0,1]$,  $P_+$ is  homeomorphic to $\mathbb S^1\times [0,1]\times [0,1]$, that is a  solid torus. So do $P_-$. We define  homeomorphisms $\Phi_\pm:\partial P_\pm\to \partial P_\pm$  with the following properties:
 \begin{enumerate}
 \item $\Psi_{\pm}|_{\partial \mathbb B^2_r\times [0,1]}=g_{{\partial \mathbb B^2_r\times [0,1]}}$;
     \item $\Psi_{+}|_{\partial \mathbb B^2\times [0,1]}={\rm id}|_{\partial \mathbb B^2\times [0,1]}$;
    \item $\Psi_{-}|_{\partial \mathbb B^2_{r_1}\times [0,1]}={\rm id}|_{\partial \mathbb B^2_{r_1}\times [0,1]}$;
     \item $\Psi_\pm$  preserves the  isotopic class of a meridian of $P_{\pm}$.
       \end{enumerate}

Let $g_{i}=g|_{\partial \mathbb B^2_r\times \{i\}}, i\in \{0,1\}$. Since $g$ is orientation preserving, so do $g_i$. Due to Statement~\ref{s-maps}, there exists an isotopy  $g^t_i: \partial \mathbb B^2_r\times \{i\}\to \partial \mathbb B^2_r\times \{i\}$  connecting $g^1_i=g_i$ with the identity map $g^0_i$.  For $x\in \mathbb B^2\setminus {\rm int}\mathbb B^2_r$ and $p\in P_+$ we set $\tau_+(|x|)=\frac{1-|x|}{1-r}$  and    

$$
G_+(p)=\begin{cases}g(p),\ x\in \partial \mathbb B^2_r\times [0,1];\cr
p,\ p\in \partial \mathbb B^2\times [0,1];\cr
\frac{|p|}{r}g^{\tau_+(|p|)}_i(\frac{rp}{|p|}),\ p\in \mathbb B^2\setminus \mathbb B^2_r\times \{i\}, i\in \{0,1\}.
\end{cases}
$$

Let  $\rho_{x}$  be a radial  segment of   $\mathbb B^2$ joining the point $x\in \partial \mathbb B^2_r$ with a point  $x'\in \partial \mathbb B^2$. Then the union 

$$\mu=l\cup \{\rho_x\times \{0,1\}\}\cup \{x'\times [0,1]\}$$

is a meridian of solid torus $P_+$.  A curve $G_+(\mu)\subset \partial P_+$ is a closed curve that intersects a canonical longitude $\lambda=\partial \mathbb B^2\times \{0\}$ of the torus $\partial P_+$ at exactly one point. Then  $[G_+(\mu)]=[\mu]+n[\lambda]$  in  a fundamental group  $\pi_1(\partial P_+)$.  If $n=0$, then   $G_+(\mu)$ is a meridian of $P_+$ and  we set $\Psi_+=G_+$.  If $n\neq 0$, there is a Dehn twist $\tau: \partial P_+\to \partial P_+$ with a support $(\mathbb B^2\setminus {\rm int}\, \mathbb B^2_r)\times \{0\}$ such that
$\tau(G_+(\mu))$ is  a meridian of $P_+$.  In this case the composition $\Psi_+=\tau G_+$ is the desired homeomorphism.  In similar way we determine the homeomorphism $\Psi_-: P_-\to P_-$. 

Last property of $\Psi_\pm$ guaranties that there exists a homeomorphism  $\Phi_\pm: P_+\to P_+$ such that $\Phi_\pm |_{\partial P_\pm}=\Psi_{\pm}|_{\partial P_+}$. Then the  desired homeomorphism $h$ is defined by the formula

\begin{equation}
h(p)=\begin{cases}\Phi_+(p), & {\rm if}\,  p\in (\mathbb B^2\setminus {\rm int}\, \mathbb B^2_{r}) \times [0,1];\cr
\Phi_-(p), & {\rm if}\, p\in (\mathbb B^2_r\setminus {\rm int}\, \mathbb B^2_{r_1}) \times [0,1];\cr
p, & {\rm if}\, p\in \mathbb B_{r_1}\times [0,1].\end{cases}
\end{equation}

\end{proofn}

 \newpage
The work was supported by the Russian Science Foundation (grant No. 23-71-30008).

\section*{CONFLICT OF INTEREST}%Note: this section is necessary

The authors declare that they have no conflicts of interest.

%%%%%%%%%%%%%%%%%%%%%%%%%%%%%

\end{document}